\documentclass[lettersize,journal]{IEEEtran}
\usepackage{amsmath,amsfonts,amssymb,amsthm}
\usepackage{algorithmic}
\usepackage{algorithm}
\usepackage{array}
\usepackage{textcomp}
\usepackage{stfloats}
\usepackage{url}
\usepackage{verbatim}
\usepackage{graphicx}
\usepackage{cite}
\usepackage{bm}
\usepackage[ntheorem]{empheq}
\usepackage{enumitem}
\usepackage{multirow}
\usepackage{booktabs}
\usepackage{hyperref}
\usepackage{subcaption}
\usepackage{stfloats}
\usepackage[ntheorem]{empheq}
\newtheorem{theorem}{Theorem}
\newtheorem{definition}[theorem]{Definition}  
\newtheorem{lemma}[theorem]{Lemma}
\newtheorem{remark}[theorem]{Remark}

\newtheorem{proposition}[theorem]{Proposition}
\newcommand{\shrink}[2]{\mathrm{shrink}(#1,\,#2)}

\begin{document}

\title{Difference-of-Convex Elastic Net for Compressed Sensing}

\author{Lang Yu, Nanjing Huang{*}
	\thanks{*Corresponding author: Nanjing Huang (e-mail: njhuang@scu.edu.cn; nanjinghuang@hotmail.com).}
	\thanks{Lang Yu and Nanjing Huang are with Department of Mathematics, Sichuan University, Sichuan, 610064, China (e-mail: langyu9602@gmail.com).}
	\thanks{This work was funded by the National Natural Science Foundation of China (12171339, 12471296).}
}



\maketitle

\begin{abstract}
This work  proposes a novel and unified sparse recovery framework, termed the difference of convex Elastic Net (DCEN). This framework effectively balances strong sparsity promotion with solution stability, and is particularly suitable for high-dimensional variable selection involving highly correlated features. Built upon a difference-of-convex (DC) structure, DCEN employs two continuously tunable parameters to unify classical and state-of-the-art models--including LASSO, Elastic Net, Ridge, and $\ell_1-\alpha\ell_2$--as special cases. Theoretically, sufficient conditions for exact and stable recovery are established under the restricted isometry property (RIP), an oracle inequality and recovery bound are derived for the global solution, and a closed-form expression of the DCEN regularization proximal operator is obtained. Moreover, two efficient optimization algorithms are developed based on the DC algorithm (DCA) and the alternating direction method of multipliers (ADMM). Within the Kurdyka-\L{}ojasiewicz (K\L{}) framework, the global convergence of DCA and its linear convergence rate are rigorously established. Furthermore, DCEN is extended to image reconstruction by incorporating total variation (TV) regularization, yielding the DCEN-TV model, which is efficiently solved via the Split Bregman method. Numerical experiments demonstrate that DCEN consistently outperforms state-of-the-art methods in sparse signal recovery, high-dimensional variable selection under strong collinearity, and Magnetic Resonance Imaging (MRI) image reconstruction, achieving superior recovery accuracy and robustness.
\end{abstract}


\section{Introduction}
In the compressed sensing (CS) framework~\cite{candes2006robust,donoho2006compressed}, the key challenge is how to choose a suitable regularization to constrain the solution for recovering a sparse signal $\bm{x}$ from underdetermined measurements $\bm{y} = \bm{\Phi}\bm{x}$. Ideally, one considers the optimization problem $\min_{\bm{x}} \|\bm{x}\|_0$ subject to $\bm{y} - \bm{\Phi}\bm{x} \in \mathcal{E}$, where $\mathcal{E} \subseteq \mathbb{R}^m$ models noise (with $\mathcal{E} = \{\mathbf{0}\}$ in the noiseless case). However, $\ell_0$ minimization is NP-hard and nonconvex. Cand\`{e}s et al.~\cite{candes2006robust} showed that if $\bm{\Phi}$ satisfies the RIP condition, $\ell_0$ minimization can be equivalently replaced by the convex $\ell_1$ minimization (Basis Pursuit, BP)~\cite{chen2001atomic}, i.e.,  $\min_{\bm{x}} \|\bm{x}\|_1$ subject to $\bm{y} - \bm{\Phi}\bm{x}\in \mathcal{E}$. Nevertheless, the $\ell_1$ approximation leads to both bias and suboptimal recovery because $\ell_1$ is dominated by large-magnitude components,  especially when $\bm{\Phi}$ has highly coherent columns~\cite{fannjiang2012coherence}.

Recent studies have shown that nonconvex sparsity-promoting measures not only significantly reduce the required number of measurements, but also achieve superior recovery performance and enhanced sparsity of the solution. Popular nonconvex measures include the $\ell_p$ quasi-norm ($p \in (0,1)$) and its variants~\cite{candes2008enhancing,chartrand2007exact,daubechies2010iteratively,xu2012l}, the minimax concave penalty \cite{zhang2010nearly}, the capped-$\ell_1$~\cite{zhang2010analysis}, the smoothly clipped absolute deviation  penalty~\cite{fan2001variable}, and the transformed $\ell_1$  ~\cite{zhang2018minimization} as well as nonconvex measures based on probability density functions~\cite{zhou2022unified}. It should be noticed that the nonconvex measures mentioned above are all separable in structure, i.e., they can be decomposed into a sum of identical one-dimensional functions. Although the separable nonconvex measures perform well in scenarios with low mutual coherence or low dynamic range, the nonseparable nonconvex measures generally exhibit superior sparse recovery performance when the sensing matrix is highly coherent or the signal possesses a high dynamic range~\cite{zhou2025recovery}. 

One important nonseparable measure is the ratio-type nonconvex measure.  Such a measure has attracted particular attention due to its scale invariance, which is getting closer as the inherently scale-invariant $\ell_0$ ``norm" (which is not a true norm). It is worth mentiong that the ratio-type nonconvex measures were first introduced in the work of Hoyer~\cite{hoyer2002non}, followed by extensive studies on the $\ell_1/\ell_2$ model~\cite{rahimi2019scale, xu2021analysis} and its associated algorithms \cite{boct2022extrapolated, tao2022minimization,wang2020accelerated,zeng2021analysis}, as well as its applications in computed tomography  reconstruction~\cite{wang2021limited}. Other ratio-type nonconvex measures include $(\ell_1/\ell_2)^2$ \cite{jia2025sparse} and the $q$-ratio sparsity measure $(\ell_1/\ell_q)^{\frac{q}{q-1}} (1<q\leq \infty)$~\cite{zhou2021minimization}. A rich body of literature has also investigated the $\ell_1/\ell_q$ ($q>1$) model, including~\cite{huang2018sparse,li2022proximal,wang2022minimizing,wang2023variant}.Despite having some advantages in approximating $\ell_0$ and promoting sparsity, ratio-type nonconvex measures suffer from highly complex structures and strong nonconvexity. Moreover, they are neither globally Lipschitz continuous nor  coercive, which render theoretical analysis and algorithm design highly challenging. In addition, some models (e.g., $\ell_1/\ell_2$) may even tend to produce an abnormally large coefficient while suppressing other nonzero components \cite{rahimi2019scale}.

Another important nonseparable nonconvex measure is the difference-type. The contour of such a measure is getting closer as the $\ell_0$ norm, which makes it more effective in promoting sparsity. The $\ell_{1-2}$ (e.g., $\ell_1-\ell_2$) measure, originally proposed in~\cite{esser2013method} as a sparsity-inducing penalty for nonnegative least squares problems and later applied to sparse recovery in~\cite{lou2015computing,yin2015minimization}, has been demonstrated to possess superior performance over existing measures when the sensing matrix is highly coherent. Subsequently, several other difference-type nonconvex measures have been proposed and extensively studied, including $\ell_1 - \alpha  \ell_2$~\cite{ge2021dantzig,lou2018fast}, $\ell_1-\beta\ell_q$~\cite{huo2023minimization}, $\ell_r - \alpha \ell_1$~\cite{zhou2023rip}, $\ell_q - \alpha \ell_p$ ($p \in (1,2]$)~\cite{gu2024nonconvex}, and weighted $\ell_q - \alpha \ell_p$ ($q \in (0,1]$, $p \in [q, +\infty]$, $\alpha \in [0,1]$)~\cite{zhan2025sparse}.  Nevertheless, difference-type nonconvex measures also exhibit inherent limitations. The first one is that the difference-type nonconvex measures may gradually become biased when the number of dominant components (ordered by magnitude) in the signal increases. For instance, the behavior of $\ell_{1-2}$ tends to approach that of $\ell_1$~\cite{ma2017truncated}. The second one is that the difference-type nonconvex measures primarily focus on sparsity itself and often fail to effectively capture the grouping effect when dealing with highly correlated variables, which may result in unstable solutions in practical applications.

To overcome the aforementioned limitations in balancing sparsity promotion and robustness under coherent sensing matrices, in the present paper, inspired by Elastic Net~\cite{zou2005regularization}, we propose a more general difference-type nonseparable nonconvex sparsity measure for compressed sensing  and image reconstruction variable selection. Specifically, we consider the following constrained optimization problem
\begin{equation}\label{eq:constprobm}
	\min_{\bm{x} \in \mathbb{R}^n} \gamma \bigl(\|\bm{x}\|_1 - \alpha \|\bm{x}\|_2\bigr) + (1-\gamma)\|\bm{x}\|_2^2 
	\quad \text{s.t.} \quad \bm{A}\bm{x} = \bm{b}, 
\end{equation}
and its unconstrained regularized form
\begin{equation}\label{eq:unconstprobm}
	\min_{\bm{x} \in \mathbb{R}^n} \frac{1}{2}\|\bm{A}\bm{x} - \bm{b}\|_2^2 
	+ \lambda \left( \gamma \bigl(\|\bm{x}\|_1 - \alpha \|\bm{x}\|_2\bigr) + (1-\gamma)\|\bm{x}\|_2^2 \right),
\end{equation}
where $\bm{A} \in \mathbb{R}^{m \times n}$ denotes the sensing matrix, $\bm{b} \in \mathbb{R}^m$ is the observation vector. The parameter $\lambda > 0$ balances the data fidelity term and the regularization term. Two key tuning parameters are introduced: $\gamma \in (0,1)$ controls the mixing ratio between the $\ell_1-\alpha\ell_2$ and $\ell_2^2$ components, while $\alpha \in (0,1)$ adjusts the strength of the nonconvex difference term. The unconstrained regularized problem~\eqref{eq:unconstprobm} is referred to as DCEN, and the constrained problem~\eqref{eq:constprobm} as CDCEN. DCEN aims to combine the high-precision recovery capability of nonconvex regularization with the stability of Elastic Net, and seeks a balance between sparsity induction and bias control through adjustable parameters. Notably, DCEN provides a highly unified mathematical framework: by flexibly tuning $\gamma$ and $\alpha$, DCEN reduces to several classical convex and nonconvex regularization models:
\begin{itemize}
	\item $\ell_1 - \ell_2$ model: when $\gamma \to 1$ and $\alpha \to 1$, the regularization term degenerates to the difference measure $\|\bm{x}\|_1 - \|\bm{x}\|_2$;
	\item $\ell_1 - \alpha \ell_2$ model: when $\gamma \to 1$ and $\alpha \in (0,1)$, DCEN recovers the classical nonconvex $\ell_1 - \alpha \ell_2$ measure;
	\item LASSO model: when $\gamma \to 1$ and $\alpha \to 0$, the $\ell_2$ term vanishes, yielding the standard $\ell_1$ minimization problem;
	\item Elastic Net model: when $\gamma \in (0,1)$ and $\alpha \to 0$, the nonconvex subtraction term disappears, reducing DCEN to Elastic Net;
	\item Ridge regression: when $\gamma \to 0$ and $\alpha \to 0$, only the $\ell_2^2$ term remains.
\end{itemize}

In summary, DCEN inherits the advantages of nonconvex measures in sparse approximation while retaining the ability of the Elastic Net to handle correlated variables. Compared with ratio-type measures, DCEN is constructed within the DC framework, which guarantees coercivity of the objective function and enables efficient optimization via DC algorithms (DCA). Tuning the parameter $\gamma$ preserves the sparsity-promoting property while effectively mitigating nonconvexity and enhancing smoothness.  

Beyond the novelty of the model, the main challenges lie in designing an optimization algorithm for the nonconvex and nonsmooth DCEN with rigorous convergence guarantees and explicit convergence rates, as well as in establishing recovery conditions for CDCEN and DCEN. These challenges necessitate novel technical tools, particularly the careful treatment of the convex–nonconvex hybrid structure jointly governed by $\gamma$ and $\alpha$. In contrast, methods based on a single sparsity measure are generally ill-suited to such hybrid structures. Moreover, our algorithmic design and theoretical analysis differ substantially from existing works \cite{yin2015minimization}, whose DC decompositions fail to satisfy the standard DCA assumptions and only guarantee sequence convergence. By contrast, to enhance sparse signal recovery while preserving the grouping effect among variables, we construct a DC decomposition satisfying the standard DCA assumptions and establish both global convergence and explicit convergence rates of the proposed algorithm. The main contributions of this paper can be summarized as follows:
\begin{enumerate}
	\item[(i)] Both constrained and unconstrained DCEN models are proposed and some conditions are obtained for ensuring the exact and stable recovery based on RIP, which extends and improvs the existing results in~\cite{yan2017sparse, yin2015minimization}. Moreover, we introduce the restricted eigenvalue condition for DCEN and use it to establish the oracle property and recovery bound for the global solution.
	\item[(ii)] A closed-form analytical solution of the proximal operator is derived for DCEN regularization under some mild assumptions, and further an efficient DC decomposition and a variable-splitting strategy are, respectively, developed for solving DCEN, which lead to DCA- and ADMM-based iteration schemes.
	\item[(iii)] Global convergence and linear convergence rates of the proposed DCA and ADMM are rigorously established under the K\L{} framework, which extend the convergence result of DCA proposed in \cite{yin2015minimization}.
	\item[(iv)] DCEN is extended to structured inverse problems by incorporating total variation regularization to form the DCEN-TV model, and an efficient algorithm based on the split Bregman method is developed, demonstrating its extensibility in image-related tasks.
	\item[(v)] Systematic experiments are provided for the noiseless or noisy signal recovery, high-dimensional correlated variable selection, and MRI reconstruction to show that the performance of DCEN consistently is superior to the state-of-the-art methods in terms of both accuracy and stability.
\end{enumerate}

The rest of this work is structured as follows. The next section presents the theoretical analysis of CDCEN. After that in Section \ref{sec:proxm_operator}, a closed-form solution of the proximal operator is given for the DCEN minimization. In Section~\ref{sec:DCAalgorithms}, we describe DCA and provide its convergence analysis. In Section \ref{sec:admm}, we develop an efficient ADMM scheme for solving the DCEN minimization. In Section \ref{sec:NumerExpms}, we report numerical experiment results, before we summarize the results in Section \ref{sec:conclusion}.

\section{Recovery properties of CDCEN}\label{sec:ReDCENprop}
This section systematically analyzes the theoretical properties of CDCEN with respect to exact recovery and robust recovery.

\subsection{Exact and stable recovery} 

We first show the following lemmas concerning the measure $\|\bm{x}\|_1 - \alpha\|\bm{x}\|_2$, which will be frequently invoked later in the proofs.
\begin{lemma}\label{lem:lowerrboun1}
	Let $\alpha \in (0,1)$ and $\bm{x} \in \mathbb{R}^n \setminus \{\bm{0}\}$ with $\|\bm{x}\|_0 = s$ and $\Lambda = \mathrm{supp}(\bm{x})$. Then the following statements hold:
	\begin{enumerate}
		\item[(a)] $(n - \alpha\sqrt{n})x_{min} + (1-\alpha)(\|\bm{x}\|_1-nx_{min}) \leq \|\bm{x}\|_1 - \alpha\|\bm{x}\|_2 \le (\sqrt{n} - \alpha) \|\bm{x}\|_2$,
		\item[(b)] $(s - \alpha\sqrt{s})x_{min} + (1-\alpha)(\|\bm{x}_{\Lambda}\|_1-sx_{min}) \leq \|\bm{x}_{\Lambda}\|_1 - \alpha\|\bm{x}_{\Lambda}\|_2 \le (\sqrt{s} - \alpha) \|\bm{x}\|_2$,
	\end{enumerate}
	where $x_{\min}$ in (a) is defined as $x_{\min} = \min_j |x_j|$ and $x_{\min}$ in (b) is defined as $x_{\min} = \min_{j \in \Lambda} |x_j|$.
\end{lemma}
\begin{proof}
	(a) By the Cauchy–Schwarz inequality $\|\bm{x}\|_1 \le \sqrt{n}\|\bm{x}\|_2$, it follows that $\|\bm{x}\|_1 - \alpha\|\bm{x}\|_2 \le \sqrt{n}\|\bm{x}\|_2 - \alpha\|\bm{x}\|_2 = (\sqrt{n}-\alpha)\|\bm{x}\|_2$.
	
	Next, we establish the lower bound. Without loss of generality, assume $|x_1| \geq |x_2| \geq \cdots \geq |x_n|$. We define the residual term of the $\ell_1$-norm as $\Delta_1 = \|\bm{x}\|_1 - n \cdot x_{\min} = \sum_{i=1}^n |x_i| - \sum_{i=1}^n x_{\min} = \sum_{i}^n (|x_i| - x_{\min})$.Since for any $ i \in \{1,2,\ldots,n\} $, we have $ |x_i| \ge x_{\min} $ and $ \Delta_1 \ge 0 $. Moreover, $ \Delta_1 = 0 $ if and only if all nonzero elements $ |x_i| $ are equal to $ x_{\min} $. Hence, the $\ell_1$-norm can be written as  
	\begin{equation}\label{eq:norm-1}
		\|\bm{x}\|_1 = n \cdot x_{\min} + \Delta_1.
	\end{equation}
	Let $ \delta_i = |x_i| - x_{\min} \ge 0 $. Then we have $ \Delta_1 = \sum_{i =1}^n \delta_i $ and
	\begin{equation}\label{eq:norm-2}
		\begin{aligned}
			\|\bm{x}\|_2^2 = n \cdot x_{\min}^2 + 2x_{\min}\Delta_1 + \|\bm{\delta}\|_2^2.
		\end{aligned}
	\end{equation}
	From the inequality $ \|\bm{\delta}\|_2 \le \|\bm{\delta}\|_1 $, one has  
	\begin{equation}\label{eq:norm-delta}
		\|\bm{\delta}\|_2^2 = \sum \delta_i^2 \le \left( \sum \delta_i \right)^2 = \Delta_1^2.
	\end{equation}
	Substituting \eqref{eq:norm-delta} into the expression of $ \|\bm{x}\|_2^2 $ in \eqref{eq:norm-2} leads to $\|\bm{x}\|_2^2 \le n \cdot x_{\min}^2 + 2x_{\min}\Delta_1 + \Delta_1^2$.
	Note that $(\sqrt{n}x_{\min} + \Delta_1)^2 = n x_{\min}^2 + 2\sqrt{n}x_{\min}\Delta_1 + \Delta_1^2$.
	Since $ n \ge 1 $ (the number of entries in $\bm{x}$ is at least one), it follows that $ \sqrt{n} \ge 1 $. By the fact that $ x_{\min} \ge 0 $ and $ \Delta_1 \ge 0 $, we have $2x_{\min}\Delta_1 \le 2\sqrt{n}x_{\min}\Delta_1$ and $n x_{\min}^2 + 2x_{\min}\Delta_1 + \Delta_1^2  \le n x_{\min}^2 + 2\sqrt{n}x_{\min}\Delta_1 + \Delta_1^2 = (\sqrt{n}x_{\min} + \Delta_1)^2$. Taking the square root on both sides gives $\sqrt{n x_{\min}^2 + 2x_{\min}\Delta_1 + \Delta_1^2} \le \sqrt{n}x_{\min} + \Delta_1$. Combining with $ \|\bm{x}\|_2^2 \le n x_{\min}^2 + 2x_{\min}\Delta_1 + \Delta_1^2 $, the upper bound of $ \|\bm{x}\|_2 $ is given by $\|\bm{x}\|_2 \le \sqrt{n}x_{\min} + \Delta_1$. Substituting the expression of $ \|\bm{x}\|_1 $ in \eqref{eq:norm-1} and the upper bound of $ \|\bm{x}\|_2 $ into $ \|\bm{x}\|_1 - \alpha\|x\|_2 $, we have  
	\begin{equation}\label{eq:normieq-half}
		\|\bm{x}\|_1 - \alpha\|\bm{x}\|_2 \ge \|\bm{x}\|_1 - \alpha(\sqrt{n}x_{\min} + \Delta_1).
	\end{equation}
	Since \( \Delta_1 = \|\bm{x}\|_1 - n x_{\min} \), it follows from  \eqref{eq:normieq-half} that  
	\begin{equation}\label{eq:dclowerr1}
		\|\bm{x}\|_1 - \alpha\|\bm{x}\|_2 \ge (n - \alpha\sqrt{n})x_{\min} + (1 - \alpha)(\|\bm{x}\|_1 - n x_{\min}).
	\end{equation}
	
	(b) It is worth noting that $\|\bm{x}\|_1 - \alpha\|\bm{x}\|_2$ depends only on the nonzero components of $ \bm{x} $, that is, $\|\bm{x}\|_1 - \alpha\|\bm{x}\|_2 = \|\bm{x}_\Lambda\|_1 - \alpha\|\bm{x}_\Lambda\|_2$.
	Hence, statement (b) directly follows by applying (a) to $ \bm{x}_\Lambda $.
\end{proof}
\begin{remark}
	The newly derived lower bound is sharper than $(n - \alpha\sqrt{n})x_{\min}$ in~\cite{li2020}.  
	Since $\|\bm{x}\|_1 - n x_{\min} = \Delta_1 \ge 0$ and $1 - \alpha > 0$,  
	the new bound becomes strictly greater than $(n - \alpha\sqrt{n})x_{\min}$  
	whenever the vector $\bm{x}$ is not flat (i.e., $\|\bm{x}\|_1 > n x_{\min}$).  
	The proposed bound explicitly exploits the residual term \(\Delta_1\),  
	which measures the excess of the $\ell_1$-norm over its minimum value $n x_{\min}$,  
	thus providing a sharper estimation.
\end{remark}

\begin{lemma}\label{lem:lowerrboun2}
	Let $\alpha \in (0,1)$ and $\bm{x} \in \mathbb{R}^n \setminus \{\bm{0}\}$ with $\|\bm{x}\|_0 = s$ and $\Lambda = \mathrm{supp}(\bm{x})$. Assume that there exist constants $ C \ge 1 $ and $ p \ge 0 $ such that $|x_i| \ge C i^{-p} \min_j |x_j|$ for all $ i \in \{1, 2, \ldots, n\} $. 	Then, the following statements hold:
	\begin{enumerate}
		\item[(a)] $C x_{\min}(\alpha n^{\frac{1}{2}-p}(n^{\frac{1}{2}}-1) + (1-\alpha) H_{n,p})\le \|\bm{x}\|_1 - \alpha\|\bm{x}\|_2$,
		\item[(b)] $C x_{\min}(\alpha s^{\frac{1}{2}-p}(s^{\frac{1}{2}}-1) + (1-\alpha) H_{s,p})\le \|\bm{x}_\Lambda\|_1 - \alpha\|\bm{x}_\Lambda\|_2$,
	\end{enumerate}
	where $x_{\min}$ is defined as the one in Lemma~\ref{lem:lowerrboun1} and $H_{n,p}=\sum_{i=1}^{n}i^{-p}$.
\end{lemma}
\begin{proof}
	(a) Without loss of generality, assume that $ |x_1| \ge |x_2| \ge \cdots \ge |x_n| $, and let $\min\{|x_i|\} = x_{\min}$. According to the assumption, we have $ |x_i| \ge C i^{-p} x_{\min} $ and so
	\begin{equation}\label{eq:L1normpower}
		\|\bm{x}\|_1 = \sum_{i=1}^n |x_i| \ge \sum_{i=1}^n C i^{-p} x_{\min} = C x_{\min} \sum_{i=1}^n i^{-p}.
	\end{equation}
	Let $ H_{n,p} = \sum_{i=1}^n i^{-p} $. Then $\|x\|_1 \ge C x_{\min} H_{n,p}$.
	Lemma~\ref{lem:lowerrboun1} and \eqref{eq:L1normpower} yield $\|\bm{x}\|_{1}-\alpha\|\bm{x}\|_2 \ge (\alpha n - \alpha \sqrt{n}) x_{\min} + (1-\alpha) C x_{\min} H_{n,p}$. It follows from $ x_{\min} \ge  C n^{-p} x_{\min}$ that $(\alpha n - \alpha \sqrt{n}) x_{\min} +\allowbreak (1-\alpha) C x_{\min} H_{n,p} \ge\allowbreak C x_{\min}\!\bigl(\alpha n^{\frac{1}{2}-p}(n^{\frac{1}{2}}-1) +\allowbreak (1-\alpha) H_{n,p}\bigr)$.
	
	(b) Statement (b) directly follows by applying (a) to $ \bm{x}_\Lambda $.
\end{proof}
\begin{remark}
	In Lemma \ref{lem:lowerrboun2}, if the parameters $ C $ and $ p $ are properly chosen,  
	the new lower bound $C x_{\min}\left(\alpha s^{\frac{1}{2}-p}\left(s^{\frac{1}{2}}-1\right) + (1-\alpha) H_{s,p}\right)$ becomes tighter than the bound $(s - \alpha \sqrt{s}) x_{\min}$.  
	For $C = s^p $ with $p\ge 0$, one has $C x_{\min}\bigl(\alpha s^{\frac{1}{2}-p}(s^{\frac{1}{2}}-1) +\allowbreak (1-\alpha) H_{s,p}\bigr) =\allowbreak x_{\min}\bigl(\alpha (s-\sqrt{s}) +\allowbreak (1-\alpha) s^p H_{s,p}\bigr)$.
	Letting	$\Delta L = x_{\min}\bigl(\alpha s - \alpha \sqrt{s} +\allowbreak (1-\alpha) s^p H_{s,p}\bigr) -\allowbreak (s - \alpha \sqrt{s}) x_{\min}$, we have	$\Delta L = (\alpha s - s +\allowbreak (1-\alpha) s^p H_{s,p}) x_{\min} =\allowbreak (1-\alpha)(s^p H_{s,p} - s) x_{\min}$. If $ 0 < \alpha < 1 $ and $ p = 0 $,  then $ s^p H_{s,0} = H_{s,0} = \sum_{i\in \Lambda} \left(\frac{s}{i}\right)^0 = s $ and $\Delta L = 0$. If $ 0 < \alpha < 1 $ and $ p > 0 $, then for any $ i \in \{1, 2, \ldots, s-1\} $,  it holds that $ s/i > 1 $ and $(s/i)^p > 1$. Thus, for any $ s \ge 2 $, one has
	\begin{equation}\label{eq:sumpower}
		\sum_{i\in\Lambda}(\frac{s}{i})^p = (\frac{s}{1})^p + \cdots + (\frac{s}{s-1})^p + (\frac{s}{s})^p > s.
	\end{equation}
	The summation in \eqref{eq:sumpower} includes $ s $ terms, each greater than or equal to 1, with at least $ s-1$ terms strictly greater than 1. Hence, $s^p H_{s,p} = s^p \sum_{i\in\Lambda} i^{-p} = \sum_{i\in\Lambda} \left(\frac{s}{i}\right)^p > s$.
	It follows that $ s^p H_{s,p} - s > 0 $ and so $ \Delta L > 0 $.
\end{remark}
For convenience, we define $\ell(\bm{x})=\frac{1}{2}\|\bm{A}\bm{x} - \bm{b}\|_2^2$ and $r(\bm{x})=\gamma\big(\|\bm{x}\|_1 - \alpha\|\bm{x}\|_2\big)
+ (1-\gamma)\|\bm{x}\|_2^2$.
\begin{lemma}\label{lem:bounderDCEN}
	Let $\alpha \in(0,1)$, $\gamma \in(0,1)$, 
	$\bm{x} \in \mathbb{R}^n \setminus \{\bm{0}\}$ with $\|\bm{x}\|_0 = s$ and support $\Lambda = \mathrm{supp}(\bm{x})$. If there exist constants $C\geq 1$ and $p \geq 0$ such that $|x_i| \geq C\, i^{-p} \min_j |x_j|$ for all $i \in \{1,2,\ldots,n\}$, then the following statements hold:
	\begin{enumerate}
		\item[(a)] If $ C = n^p$ and $p \geq 0$, then 
		\begin{equation}
			\begin{aligned}
				&\gamma \big( \alpha n - \alpha \sqrt{n} + (1-\alpha)n^p H_{n,p} \big) x_{\min} + (1-\gamma)\|\bm{x}\|_2^2  \\
				&\leq r(\bm{x}) \leq \gamma(\sqrt{n} - \alpha)\|\bm{x}\|_2 + (1-\gamma)\|\bm{x}\|_2^2.
			\end{aligned}
		\end{equation}
		\item[(b)] If $\bm{x}$ has sparsity $s$ and $ C = s^p$ with $p \geq 0$, then
		\begin{equation}
			\begin{aligned}
				&\gamma \big( \alpha s - \alpha \sqrt{s} + (1-\alpha)s^p  H_{s,p} \big) x_{\min}+ (1-\gamma)\|\bm{x}\|_2^2 \\
				& \leq r(\bm{x}) \leq \gamma(\sqrt{s} - \alpha)\|\bm{x}\|_2 + (1-\gamma)\|\bm{x}\|_2^2.
			\end{aligned}
		\end{equation}
		\item[(c)] If $\bm{x} $ has sparsity $s$ and $\|\bm{x}\|_2 = r$, then $r(\bm{x}) > 0$ and $r(\bm{x})$ attains its minimum when $s = 1$.
	\end{enumerate}
\end{lemma}
\begin{proof}
	It follows from $\alpha,\gamma \in(0,1)$ that the statements (a) and (b) can be readily verified from Lemma~\ref{lem:lowerrboun2}. 
	We now prove the statement (c). Since $\alpha \in(0,1)$ and $\gamma \in(0,1)$, it follows that $s - \alpha \sqrt{s} \geq 0$ and so $r(\bm{x}) \geq \gamma (s - \alpha \sqrt{s}) x_{\min} + (1-\gamma)\|\bm{x}\|_2^2 > 0$, which shows that $r(\bm{x}) > 0$ for all $\bm{x} \neq 0$. Since $\|\bm{x}\|_2 = r$, one has $r(\bm{x}) = \gamma \|\bm{x}\|_1 - \gamma \alpha r + (1-\gamma) r^2 = \gamma (\|\bm{x}\|_1 - \alpha r) + (1-\gamma) r^2$ and so the minimum of $\|\bm{x}\|_1$ is achieved when the support size satisfies $|\mathrm{supp}(\bm{x})| = s = 1$. In this case, the vector $\bm{x}$ has only one nonzero entry of magnitude $r$, so $\|\bm{x}\|_1 = r$. Hence, $r(\bm{x})$ attains its minimum when $s = 1$.
\end{proof}

Building upon the RIP-based analysis framework in~\cite{yin2015minimization,candes2005error}, which established sufficient conditions for the exact recovery of the $\ell_{1}-\ell_{2}$ minimization and BP problem, we extend such a framework to CDCEN under the distinct condition.
\begin{theorem}\label{eq:ExtraRecovery}
	Let $\alpha \in (0,1)$, $\gamma \in (0,1)$, and $p \geq 0$.  Suppose that all the assumptions of Lemma~\ref{lem:lowerrboun2} hold with $C = (3s)^p$. 
	Let $\bm{\bar{x}} \in \mathbb{R}^n \setminus \{\bm{0}\}$ be an $s$-sparse signal, i.e., $\|\bm{\bar{x}}\|_0 = s$ and $\|\bm{\bar{x}}\|_2\leq R (R>0)$, such that
	\begin{equation*}
		a(s,\gamma,\alpha,p) = \left(\frac{\gamma(\alpha\sqrt{3s} - \alpha + p(\alpha,s))}{\gamma\sqrt{s} + \gamma\alpha + 2(1-\gamma) R}\right)^2 > 1,
	\end{equation*}
	where $p(\alpha,s) = (1-\alpha)(3s)^p H_{3s,p} / \sqrt{3s}$. If the sensing matrix $\bm{A}$ satisfies the RIP condition $\delta_{3s} + a(s,\gamma,\alpha,p)\,\delta_{4s} < a(s,\gamma,\alpha,p) - 1$ and $\bm{A}\bm{\bar{x}} = \bm{b}$, then $\bm{\bar{x}}$ is a unique solution to the optimization problem \eqref{eq:constprobm}.
\end{theorem}
\begin{proof}
	Let $\bm{x}$ be any feasible solution of \eqref{eq:constprobm} satisfying the linear constraint $\bm{Ax} = \bm{b}$ and has a smaller or equal function value, i.e., $r(\bm{x}) \le r(\bar{\bm{x}})$. Decomposing $\bm{x}$ as $\bm{x} = \bar{\bm{x}} + \bm{h}$ with $\bm{h} \in \ker(\bm{A})$, we show that $\bm{h} = \bm{0}$. To this end, let $\Lambda = \operatorname{supp}(\bar{\bm{x}})$ denote the support set of $\bar{\bm{x}}$, and decompose $\bm{h}$ into $\bm{h} = \bm{h}_{\Lambda} + \bm{h}_{\Lambda^c}$, 
	where $\Lambda^c$ denotes the complement of $\Lambda$. Then
	\begin{equation}\label{eq:term0}
		\begin{aligned}
			\gamma (\|\bm{x}\|_1 &- \alpha \|\bm{x}\|_2) + (1-\gamma)\|\bm{x}\|_2^2\\ &\le \gamma (\|\bar{\bm{x}}\|_1 - \alpha\|\bar{\bm{x}}\|_2) + (1-\gamma)\|\bar{\bm{x}}\|_2^2.
		\end{aligned}
	\end{equation}
	For the $\ell_1$-norm term $\|\bar{\bm{x}} + \bm{h}_{\Lambda} + \bm{h}_{\Lambda^c}\|_1$, one has
	\begin{equation}\label{eq:term1}
		\begin{aligned}
			\|\bar{\bm{x}} + \bm{h}_{\Lambda} + \bm{h}_{\Lambda^c}\|_1 &= \|\bar{\bm{x}} + \bm{h}_{\Lambda}\|_1 + \|\bm{h}_{\Lambda^c}\|_1 \\
			&\ge \|\bar{\bm{x}}\|_1 - \|\bm{h}_{\Lambda}\|_1 + \|\bm{h}_{\Lambda^c}\|_1.
		\end{aligned}
	\end{equation}
	Since $\alpha \in (0,1)$, the $\ell_2$-norm term $\alpha\|\bar{\bm{x}} + \bm{h}_{\Lambda} + \bm{h}_{\Lambda^c}\|_2$ satisfies
	\begin{equation}\label{eq:term2}
		\alpha \|\bar{\bm{x}} + \bm{h}_{\Lambda} + \bm{h}_{\Lambda^c}\|_2 \le \alpha (\|\bar{\bm{x}}\|_2 + \|\bm{h}_{\Lambda}\|_2 + \|\bm{h}_{\Lambda^c}\|_2).
	\end{equation}
	Moreover,
	\begin{equation}\label{eq:term3}
		(1-\gamma)\|\bar{\bm{x}} + \bm{h}_{\Lambda} + \bm{h}_{\Lambda^c}\|_2^2 = (1-\gamma)(\|\bar{\bm{x}}+\bm{h}_{\Lambda}\|_2^2 + \|\bm{h}_{\Lambda^c}\|_2^2).
	\end{equation}
	Substituting \eqref{eq:term1},\eqref{eq:term2},and \eqref{eq:term3} into the inequality \eqref{eq:term0} yields
	\begin{equation}\label{eq:term4}
		\begin{aligned}
			&\gamma (\|\bar{\bm{x}} + \bm{h}_{\Lambda} + \bm{h}_{\Lambda^c}\|_1 - \alpha\|\bar{\bm{x}} + \bm{h}_{\Lambda} + \bm{h}_{\Lambda^c}\|_2) +\\ &(1-\gamma)\|\bar{\bm{x}} + \bm{h}_{\Lambda} + \bm{h}_{\Lambda^c}\|_2^2 \ge \gamma [\|\bar{\bm{x}}\|_1 - \|\bm{h}_{\Lambda}\|_1 \\
				& + \|\bm{h}_{\Lambda^c}\|_1 - \alpha(\|\bar{\bm{x}}\|_2 + \|\bm{h}_{\Lambda}\|_2 + \|\bm{h}_{\Lambda^c}\|_2)] \\
			& \quad \quad \quad \quad \;\;+ (1-\gamma)(\|\bar{\bm{x}}+\bm{h}_{\Lambda}\|_2^2 + \|\bm{h}_{\Lambda^c}\|_2^2).
		\end{aligned}
	\end{equation}
	It follows from \eqref{eq:term4} that $\gamma \left[\|\bar{\bm{x}}\|_1 - \|\bm{h}_{\Lambda}\|_1 + \|\bm{h}_{\Lambda^c}\|_1 - \alpha(\|\bar{\bm{x}}\|_2 + \|\bm{h}_{\Lambda}\|_2 + \|\bm{h}_{\Lambda^c}\|_2)\right]+ (1-\gamma)(\|\bar{\bm{x}}+\bm{h}_{\Lambda}\|_2^2 + \|\bm{h}_{\Lambda^c}\|_2^2) \leq \gamma (\|\bar{\bm{x}}\|_1 - \alpha\|\bar{\bm{x}}\|_2) + (1-\gamma)\|\bar{\bm{x}}\|_2^2$. Simplifying and rearranging yields $\gamma\left[-\|\bm{h}_{\Lambda}\|_1 + \|\bm{h}_{\Lambda^c}\|_1 - \alpha(\|\bm{h}_{\Lambda}\|_2 + \|\bm{h}_{\Lambda^c}\|_2)\right] + (1-\gamma)(\|\bar{\bm{x}}+\bm{h}_{\Lambda}\|_2^2 + \|\bm{h}_{\Lambda^c}\|_2^2) \le (1-\gamma)\|\bar{\bm{x}}\|_2^2$. By the Cauchy–Schwarz inequality, we have $\|\bar{\bm{x}} + \bm{h}_{\Lambda}\|_2^2 \ge (\|\bar{\bm{x}}\|_2 - \|\bm{h}_{\Lambda}\|_2)^2$ and so
	\begin{equation}
		\begin{aligned}
			\gamma\bigl[ & -\|\bm{h}_{\Lambda}\|_1 + \|\bm{h}_{\Lambda^c}\|_1 - \alpha(\|\bm{h}_{\Lambda}\|_2 + \|\bm{h}_{\Lambda^c}\|_2) \bigr] \\ 
			&+  (1-\gamma)(\|\bar{\bm{x}}+\bm{h}_{\Lambda}\|_2^2 + \|\bm{h}_{\Lambda^c}\|_2^2) \\
			&\geq \gamma\bigl[ -\|\bm{h}_{\Lambda}\|_1 + \|\bm{h}_{\Lambda^c}\|_1 - \alpha(\|\bm{h}_{\Lambda}\|_2 + \|\bm{h}_{\Lambda^c}\|_2) \bigr] \\
			& \quad +(1-\gamma)\bigl( \|\bar{\bm{x}}\|_2^2 + \|\bm{h}_{\Lambda^c}\|_2^2 - 2\|\bar{\bm{x}}\|_2\|\bm{h}_{\Lambda^c}\|_2 \bigr).
		\end{aligned}
	\end{equation}
	Simplifying and rearranging yields $\gamma \|\bm{h}_{\Lambda^c}\|_1 - \gamma \alpha \|\bm{h}_{\Lambda^c}\|_2 + (1-\gamma)\|\bm{h}_{\Lambda^c}\|_2^2 \le \gamma \|\bm{h}_{\Lambda}\|_1 + \gamma \alpha \|\bm{h}_{\Lambda}\|_2 + (\gamma - 1)(\|\bm{h}_{\Lambda}\|_2^2 - 2\|\bar{\bm{x}}\|_2\|\bm{h}_{\Lambda}\|_2)$. Since $0 < \gamma < 1$, we further obtain
	\begin{equation}\label{eq:finbound}
		\begin{aligned}
			\|\bm{h}_{\Lambda^c}\|_1 &- \alpha\|\bm{h}_{\Lambda^c}\|_2 + \frac{1-\gamma}{\gamma}(\|\bm{h}_{\Lambda^c}\|_2^2+ \|\bm{h}_{\Lambda}\|_2^2 \\ 
			&- 2\|\bar{\bm{x}}\|_2\|\bm{h}_{\Lambda}\|_2) 	\le \|\bm{h}_{\Lambda}\|_1 + \alpha\|\bm{h}_{\Lambda}\|_2.
		\end{aligned}
	\end{equation}
	We sort the components of $\Lambda^c$ in descending order by magnitude and divide them into disjoint subsets $\Lambda_i^c$ with cardinality $3s$ (the last one possibly smaller), i.e.,
	$\Lambda^c = \bigcup_{i=1}^{\ell}\Lambda_i^c$ and $\Lambda_i^c \cap \Lambda_j^c = \emptyset$ for $i\ne j$. Let $\Lambda_0^c = \Lambda \cup \Lambda_1^c$. Using the RIP property of $\bm{A}$ and $\bm{A}\bm{h} = 0$, we have	$0 = \|\bm{Ah}\|_2  \geq \left\| \bm{A}_{\Lambda_0^c} \bm{h}_{\Lambda_0^c} \right\|_2 - \left\| \sum_{i=2}^\ell \bm{A}_{\Lambda_i^c} \bm{h}_{\Lambda_i^c} \right\|_2 \geq \left\| \bm{A}_{\Lambda_0^c} \bm{h}_{\Lambda_0^c} \right\|_2 - \sum_{i=2}^\ell \left\| \bm{A}_{\Lambda_i^c} \bm{h}_{\Lambda_i^c} \right\|_2$. Set  $A \coloneqq \sqrt{1 - \delta_{4s}}$ and $B \coloneqq \sqrt{1 + \delta_{3s}}$. It follows from the RIP condition that $0 = \|\bm{Ah}\|_2 \ge A\|\bm{h}_{\Lambda_0^c}\|_2 - B\sum_{i=2}^{\ell}\|\bm{h}_{\Lambda_i^c}\|_2$. By the statement (b) in Lemma~\ref{lem:bounderDCEN}, for any $i \ge 2$ and $t \in \Lambda_i^c$, one has
	\begin{equation*}
		\begin{aligned}
			|h_t| + &\frac{(1-\gamma)\|\bm{h}_{\Lambda_i^c}\|_2^2}{\gamma(3\alpha s - \alpha\sqrt{3s} + (1-\alpha)3^ps^pH_{3s,p})}\\
			&\leq \frac{\gamma\left( \|\bm{h}_{\Lambda_{i-1}^c}\|_2 - \alpha \|\bm{h}_{\Lambda_{i-1}^c}\|_2 \right) + (1-\gamma)\|\bm{h}_{\Lambda_{i-1}^c}\|_2^2}{\gamma(3\alpha s - \alpha\sqrt{3s} + (1-\alpha)3^ps^pH_{3s,p})}.
		\end{aligned}
	\end{equation*}
	Thus, for any $i \ge 2$, $\|\bm{h}_{\Lambda_i^c}\|_2^2$ admits the upper bound
	\begin{equation*}
		\|\bm{h}_{\Lambda_i^c}\|_2^2 \le 3s\left[\frac{\gamma(\|\bm{h}_{\Lambda_{i-1}^c}\|_1 - \alpha\|\bm{h}_{\Lambda_{i-1}^c}\|_2) + (1-\gamma)\|\bm{h}_{\Lambda_{i-1}^c}\|_2^2}{\gamma(3\alpha s - \alpha \sqrt{3s} + (1-\alpha)3^ps^p H_{3s,p})}\right]^2.
	\end{equation*}
	Summing over $i \ge 2$ and simplifying yields
	\begin{equation}\label{eq:simpsum}
		\begin{aligned}
			&\sum_{i=2}^\ell \|\bm{h}_{\Lambda_i^c}\|_2\\
			&\leq \sqrt{3s} \cdot \frac{ \gamma \sum_{i=1}^{\ell} \|\bm{h}_{\Lambda_i^c}\|_1 - \gamma \alpha \ell_{h^c} + (1-\gamma) \sum_{i=1}^{\ell} \|\bm{h}_{\Lambda_i^c}\|_2^2 }{ \gamma \left( 3\alpha s - \alpha\sqrt{3s} + (1-\alpha) \, 3^p s^p H_{3s,p} \right) },
		\end{aligned}
	\end{equation}
	where $\ell_{h^c} = \sum_{i=1}^{\ell} \|\bm{h}_{\Lambda_i^c}\|_2$. By $\sum_{i=1}^\ell \|\bm{h}_{\Lambda_i^c}\|_1 = \|\bm{h}_{\Lambda^c}\|_1$, $(1-\gamma) \sum_{i=1}^\ell \|\bm{h}_{\Lambda_i^c}\|_2^2 = (1-\gamma) \|\bm{h}_{\Lambda^c}\|_2^2$, and $\gamma \alpha \sum_{i=1}^\ell \|\bm{h}_{\Lambda_i^c}\|_2 \geq \gamma \alpha \sqrt{ \sum_{i=1}^\ell \|\bm{h}_{\Lambda_i^c}\|_2^2 } = \gamma \alpha \|\bm{h}_{\Lambda^c}\|_2$, and substitute \eqref{eq:finbound} into \eqref{eq:simpsum} to obtain 
	\begin{equation*}
		\begin{aligned}
			\sum_{i=2}^\ell \|\bm{h}_{\Lambda_i^c}\|_2 
			\overset{b}{\leq} \frac{ \left( \gamma\sqrt{s} + \gamma\alpha + 2(1-\gamma) R \right) \|\bm{h}_\Lambda\|_2 }{ \gamma \left( \alpha \sqrt{3s} - \alpha + (1-\alpha) 3^ps^p H_{3s,p} / \sqrt{3s} \right) },
		\end{aligned}
	\end{equation*}
	where inequality $\overset{a}{\leq}$ holds due to condition $\gamma \|\bm{h}_{\Lambda}\|_1 \leq \gamma \sqrt{s}\|\bm{h}_{\Lambda}\|_2$ and inequality $\overset{b}{\leq}$ follows from the boundedness of the signal $\bm{\bar{x}}$, i.e., $\|\bm{\bar{x}}\|_2 \le R$. Let
	\begin{equation*}
		\sqrt{a(s,\gamma,\alpha,p)} = \frac{\gamma(\alpha\sqrt{3s} - \alpha + (1-\alpha)(3s)^p H_{3s,p} / \sqrt{3s})}{\gamma\sqrt{s} + \gamma\alpha + 2(1-\gamma) R}.
	\end{equation*}
	Then $\sum_{i=2}^\ell \|\bm{h}_{\Lambda_i^c}\|_2 \leq \frac{\|\bm{h}_\Lambda\|_2}{\sqrt{a(s, \gamma, \alpha, p)}}$. 	It follows that $0 \geq A \|\bm{h}_{\Lambda_0^c}\|_2 - B \sum_{i=2}^\ell \|\bm{h}_{\Lambda_i^c}\|_2 \geq A \|\bm{h}_{\Lambda_0^c}\|_2 - \frac{B}{\sqrt{a(s, \gamma, \alpha, p)}} \|\bm{h}_\Lambda\|_2 \overset{a}{\geq} \left( A - \frac{B}{\sqrt{a(s, \lambda, \alpha, p)}} \right) \|\bm{h}_{\Lambda_0^c}\|_2$, where the inequality $\overset{a}{\geq}$ holds because $\|\bm{h}_{\Lambda_0^c}\|_2^2 = \|\bm{h}_\Lambda + \bm{h}_{\Lambda_1^c}\|_2^2 \geq \|\bm{h}_\Lambda\|_2^2$. Since $a(s,\gamma,\alpha,p)>1$ and $\delta_{3s} + a(s, \gamma, \alpha, p)\delta_{4s} < a(s, \gamma, \alpha, p) - 1$, it follows that $\sqrt{1-\delta_{4s}} - \frac{\sqrt{1+\delta_{3s}}}{\sqrt{a(s,\lambda,\alpha,p)}} > 0$, which implies $\bm{h}_{\Lambda_0^c} = \bm{0}$ and so $\bm{h} = \bm{0}$. This proves the uniqueness of $\bar{\bm{x}}$.
\end{proof}

Theorem \ref{eq:ExtraRecovery} establishes a RIP–based sufficient condition ensuring the uniqueness of the solution recovered by the measure $\gamma\left(\|\bm{x}\|_1 - \alpha\|\bm{x}\|_2\right) + \left(1-\gamma\right)\|\bm{x}\|_2^2$. This result parallels the classical RIP condition for exact recovery in BP~\cite{candes2005error}, yet it is specifically tailored to the composite nonconvex objective above. Unlike the $\ell_1$ formulation in BP, this condition explicitly characterizes the interplay between the sparsity-inducing term $\|\bm{x}\|_1 - \alpha\|\bm{x}\|_2$ and the quadratic regularization $\|\bm{x}\|_2^2$. 

At the end of this subsection,  we extend the stable recovery property of \cite{yin2015minimization} to demonstrate the stable recovery property of CDCEN in the presence of measurement noise.
\begin{theorem}\label{thm:roubsterrbound}
	Let $\alpha \in (0,1)$ and $\gamma \in (0,1)$, and suppose that all the conditions in Theorem~\ref{eq:ExtraRecovery} are satisfied with $ \bm{A}\bm{\bar{x}} + \bm{e}=\bm{b}$,  where $\bm{e} \in \mathbb{R}^m$ denotes an arbitrary noise with $\|\bm{e}\|_2 \leq \tau$. Let $\bm{x}^{\mathrm{opt}}$ be an optimal solution of the noise-aware problem
	\begin{equation*}\label{eq:robust_dcen}
		\min_{\bm{x} \in \mathbb{R}^n}
		r(\bm{x}) \quad \text{subject to} \quad
		\|\bm{A}\bm{x} - \bm{b}\|_2 \leq \tau .
	\end{equation*}
	Then the reconstruction error admits the stability bound $\|\bm{x}^{\mathrm{opt}} - \bm{\bar{x}}\|_2 \leq C_s \, \tau$, where the constant $C_s$ is given by $C_s := \frac{2\sqrt{1 + a(s, \gamma, \alpha, p)}}{\sqrt{a(s, \gamma, \alpha, p)(1 - \delta_{4s})} - \sqrt{1 + \delta_{3s}}} > 0$. 
\end{theorem}
\begin{proof}
	The proof is similar to the one of \cite[Theorem 2.2]{yin2015minimization} and
	so we omit it here.
\end{proof}

\section{Consistency Theory for DCEN}
In this section, we study the consistency theory of DCEN. Let $\bar{\bm{x}}$ be the ground true $s$-sparse solution with support $\mathcal{S} := \{i : \bar{x}_i \neq 0\}$ and $s := |\mathcal{S}|$. We define the level set of the objective function at $\bar{\bm{x}}$ as $\text{lev}_h(\bar{\bm{x}}) := \{\bm{x} \in \mathbb{R}^n : h(\bm{x}) \le h(\bar{\bm{x}})\}$. Associated with the proposed regularization, we introduce an extended notion of the Restricted Eigenvalue Condition (REC).

\begin{definition}[DCEN-REC$(s, t, \beta)$]\label{def:extended_rec}
	Let $\beta >0$. The matrix $\bm{A} \in \mathbb{R}^{m \times n}$ is said to satisfy the restricted eigenvalue condition relative to the DCEN penalty and $(s, t)$, denoted as DCEN-REC$(s, t, \beta)$, if
	\begin{equation}
		\begin{split}
			\phi_\beta(s, t) := \min \bigg\{ & \frac{\|\bm{A}\bm{x}\|_2}{\|\bm{x}_{\mathcal{J}}\|_2} : |\mathcal{I}| \le s, \|\bm{x}_{\mathcal{I}^c}\|_1 \le \|\bm{x}_{\mathcal{I}}\|_1  \\
			& \quad + \beta \|\bm{x}\|_2, \mathcal{J} = \mathcal{I} \cup \mathcal{I}(\bm{x}; t) \bigg\} > 0.
		\end{split}
	\end{equation}
	where $\mathcal{I}(\bm{x};t)$ denotes the index set of the first $t$ largest coordinates in absolute value of $\bm{x}$ in $\mathcal{I}^c$.
\end{definition}
The following theorem provides the oracle inequality and recovery bound for DCEN.

\begin{theorem}[]\label{thm:main_bound}
	Let $(\bar{\bm{x}}, \mathcal{S}, s)$ be defined as above. Assume the ground true signal is bounded such that $\|\bar{\bm{x}}\|_2 \le R$ for some constant $R > 0$. We choose the regularization parameter $\gamma$ to satisfy the structural matching condition $\gamma \ge \frac{2R}{2R + \sqrt{s}}$. Let $\bm{x}^* \in \text{lev}_h(\bar{\bm{x}})$, and suppose that $\bm{A}$ satisfies the DCEN-REC$(s, s, \beta)$ with $\beta := 2\alpha + \sqrt{s}$. Define $\Gamma := \lambda\gamma(2\sqrt{s} + \alpha)$.	Then the following statements hold:
	\begin{enumerate}
		\item[(i)] The oracle inequality holds:
		\begin{equation}\label{eq:oracle}
			\ell(\bm{x}^*) + \lambda\gamma \bigl( \|\bm{x}^*_{\mathcal{S}^c}\|_1 - \alpha\|\bm{x}^*_{\mathcal{S}^c}\|_2 \bigr) \le \frac{2 \Gamma^2}{\phi_\beta^2(s,s)}.
		\end{equation}
		
		\item[(ii)] The recovery bound holds:
		\begin{equation}\label{eq:recovery}
			\|\bm{x}^* - \bar{\bm{x}}\|_2^2 \le \frac{\Gamma^2}{\phi_\beta^4(s,s)} \left( 4 + \frac{\Gamma^2}{4s \lambda^2 \gamma^2 (1-\alpha)^2} \right).
		\end{equation}
	\end{enumerate}
\end{theorem}

\begin{proof}
	(i) Let $\bm{x}^* \in \text{lev}_h(\bar{\bm{x}})$. Then $h(\bm{x}^*) \le h(\bar{\bm{x}})$. Denoting the residual as $\hat{\bm{x}} := \bm{x}^* - \bar{\bm{x}}$ and recalling $\bm{b} = \bm{A}\bar{\bm{x}}$, we have
	\begin{equation}\label{eq:proof_start}
		\frac{1}{2}\|\bm{A}\hat{\bm{x}}\|_2^2 + \lambda r(\bm{x}^*)	\le \lambda r(\bar{\bm{x}}).
	\end{equation}
	By the subadditivity of norms and the fact that $\bar{\bm{x}}_{\mathcal{S}^c} = \bm{0}$, which implies $\bm{x}^*_{\mathcal{S}^c} = \hat{\bm{x}}_{\mathcal{S}^c}$.Then, we have
	\begin{equation}\label{eq:rel_l1}
		\begin{aligned}
			\|\bar{\bm{x}}\|_1 - \|\bm{x}^*\|_1 &= \|\bar{\bm{x}}_{\mathcal{S}}\|_1 - (\|\bm{x}^*_{\mathcal{S}}\|_1 + \|\bm{x}^*_{\mathcal{S}^c}\|_1)\\
			 &\le \|\hat{\bm{x}}_{\mathcal{S}}\|_1 - \|\hat{\bm{x}}_{\mathcal{S}^c}\|_1
		\end{aligned}
	\end{equation}
	and $-\alpha\|\bm{x}^*\|_2 \ge -\alpha(\|\bm{x}^*_{\mathcal{S}}\|_2 + \|\bm{x}^*_{\mathcal{S}^c}\|_2) = -\alpha\|\bm{x}^*_{\mathcal{S}}\|_2 - \alpha\|\hat{\bm{x}}_{\mathcal{S}^c}\|_2$.
	Moreover, the reverse triangle inequality yields $\|\bm{x}^*_{\mathcal{S}}\|_2 - \|\bar{\bm{x}}\|_2 \le \|\bm{x}^*_{\mathcal{S}} - \bar{\bm{x}}_{\mathcal{S}}\|_2 = \|\hat{\bm{x}}_{\mathcal{S}}\|_2$, which leads to $\alpha\|\bm{x}^*_{\mathcal{S}}\|_2 - \alpha\|\bar{\bm{x}}\|_2 \le \alpha\|\hat{\bm{x}}_{\mathcal{S}}\|_2$. Applying the Cauchy-Schwarz inequality, one has $\|\bar{\bm{x}}\|_2^2 - \|\bm{x}^*\|_2^2 = -2\langle \bar{\bm{x}}, \hat{\bm{x}} \rangle - \|\hat{\bm{x}}\|_2^2 \le 2\|\bar{\bm{x}}\|_2 \|\hat{\bm{x}}_{\mathcal{S}}\|_2 - \|\hat{\bm{x}}\|_2^2$.
	Substituting these relations into \eqref{eq:proof_start} and rearranging yields
	\begin{equation}\label{eq:proof_mid1}
		\begin{aligned}
			&\frac{1}{2}\|\bm{A}\hat{\bm{x}}\|_2^2 + \lambda(1-\gamma)\|\hat{\bm{x}}\|_2^2 + \lambda\gamma\bigl(\|\hat{\bm{x}}_{\mathcal{S}^c}\|_1 - \alpha\|\hat{\bm{x}}_{\mathcal{S}^c}\|_2\bigr)\\ & \quad \le \lambda\gamma \|\hat{\bm{x}}_{\mathcal{S}}\|_1 + \lambda\gamma\alpha\|\hat{\bm{x}}_{\mathcal{S}}\|_2 + 2\lambda(1-\gamma)\|\bar{\bm{x}}\|_2 \|\hat{\bm{x}}_{\mathcal{S}}\|_2.
		\end{aligned}
	\end{equation}
	It follows from $\|\bar{\bm{x}}\|_2 \le R$ and $\gamma \ge \frac{2R}{2R + \sqrt{s}}$ that
	\begin{equation}\label{eq:bunea_trick}
		2\lambda(1-\gamma)\|\bar{\bm{x}}\|_2 \|\hat{\bm{x}}_{\mathcal{S}}\|_2 \le 2\lambda(1-\gamma)R \|\hat{\bm{x}}_{\mathcal{S}}\|_2 \le \lambda\gamma\sqrt{s} \|\hat{\bm{x}}_{\mathcal{S}}\|_2.
	\end{equation}
	Substituting \eqref{eq:bunea_trick} back into the RHS of \eqref{eq:proof_mid1}, we obtain
	\begin{equation}\label{eq:proof_mid_simplified}
		\begin{aligned}
			&\frac{1}{2}\|\bm{A}\hat{\bm{x}}\|_2^2 + \lambda(1-\gamma)\|\hat{\bm{x}}\|_2^2 + \lambda\gamma\bigl(\|\hat{\bm{x}}_{\mathcal{S}^c}\|_1 - \alpha\|\hat{\bm{x}}_{\mathcal{S}^c}\|_2\bigr)\\ &\le \lambda\gamma \|\hat{\bm{x}}_{\mathcal{S}}\|_1 + \lambda\gamma(\alpha + \sqrt{s})\|\hat{\bm{x}}_{\mathcal{S}}\|_2.
		\end{aligned}
	\end{equation}
	Then, one has $\|\hat{\bm{x}}_{\mathcal{S}^c}\|_1 \le \|\hat{\bm{x}}_{\mathcal{S}}\|_1 + \alpha\|\hat{\bm{x}}_{\mathcal{S}^c}\|_2 + (\alpha + \sqrt{s})\|\hat{\bm{x}}_{\mathcal{S}}\|_2$.
	Since both $\|\hat{\bm{x}}_{\mathcal{S}^c}\|_2 \le \|\hat{\bm{x}}\|_2$ and $\|\hat{\bm{x}}_{\mathcal{S}}\|_2 \le \|\hat{\bm{x}}\|_2$, it strictly follows that
	\begin{equation*}
		\|\hat{\bm{x}}_{\mathcal{S}^c}\|_1 \le \|\hat{\bm{x}}_{\mathcal{S}}\|_1 + (2\alpha + \sqrt{s})\|\hat{\bm{x}}\|_2 = \|\hat{\bm{x}}_{\mathcal{S}}\|_1 + \beta \|\hat{\bm{x}}\|_2.
	\end{equation*}
	Thus, $\hat{\bm{x}}$ geometrically falls into the restricted cone, and the condition $\phi_\beta(s,s)$ is applicable, implying $\|\hat{\bm{x}}_{\mathcal{S}}\|_2 \le \frac{1}{\phi_\beta(s,s)} \|\bm{A}\hat{\bm{x}}\|_2$. 
	It follows from $\|\hat{\bm{x}}_{\mathcal{S}}\|_1 \le \sqrt{s}\|\hat{\bm{x}}_{\mathcal{S}}\|_2$ that
	\begin{equation*}
		\begin{aligned}
			\lambda\gamma (\|\hat{\bm{x}}_{\mathcal{S}}\|_1 + (\alpha + \sqrt{s})\|\hat{\bm{x}}_{\mathcal{S}}\|_2)
			 \le \lambda\gamma(2\sqrt{s} + \alpha)\|\hat{\bm{x}}_{\mathcal{S}}\|_2.
		\end{aligned}
	\end{equation*}
	Dropping the $\lambda(1-\gamma)\|\hat{\bm{x}}\|_2^2$ term on the LHS of \eqref{eq:proof_mid_simplified} gives
	\begin{equation}\label{eq:proof_mid2}
		\frac{1}{2}\|\bm{A}\hat{\bm{x}}\|_2^2 + \lambda\gamma\bigl(\|\bm{x}^*_{\mathcal{S}^c}\|_1 - \alpha\|\bm{x}^*_{\mathcal{S}^c}\|_2\bigr) \le \frac{\Gamma}{\phi_\beta(s,s)} \|\bm{A}\hat{\bm{x}}\|_2.
	\end{equation}
	Applying the inequality $ab \le \frac{1}{2}a^2 + \frac{1}{2}b^2$ to the RHS of \eqref{eq:proof_mid2} with $a = \|\bm{A}\hat{\bm{x}}\|_2$ and $b = \frac{\Gamma}{\phi_\beta(s,s)}$, we obtain
	\begin{equation*}
		\frac{1}{2}\|\bm{A}\hat{\bm{x}}\|_2^2 + \lambda\gamma\bigl(\|\bm{x}^*_{\mathcal{S}^c}\|_1 - \alpha\|\bm{x}^*_{\mathcal{S}^c}\|_2\bigr) \le \frac{1}{2}\|\bm{A}\hat{\bm{x}}\|_2^2 + \frac{\Gamma^2}{2\phi_\beta^2(s,s)}.
	\end{equation*}
	Subtracting $\frac{1}{2}\|\bm{A}\hat{\bm{x}}\|_2^2$ temporarily gives $\lambda\gamma\bigl(\|\bm{x}^*_{\mathcal{S}^c}\|_1 - \alpha\|\bm{x}^*_{\mathcal{S}^c}\|_2\bigr) \le \frac{\Gamma^2}{2\phi_\beta^2(s,s)}$. Re-substituting this bound back into \eqref{eq:proof_mid2} yields the oracle inequality \eqref{eq:oracle}.
	
	(ii) Since $\|\bm{x}^*_{\mathcal{S}^c}\|_2 \le \|\bm{x}^*_{\mathcal{S}^c}\|_1$, we have the strict lower bound for the nonconvex penalty $\|\bm{x}^*_{\mathcal{S}^c}\|_1 - \alpha\|\bm{x}^*_{\mathcal{S}^c}\|_2 \ge (1-\alpha)\|\bm{x}^*_{\mathcal{S}^c}\|_1$. Combining this with \eqref{eq:oracle}, we explicitly bound the tail $\ell_1$ norm
	\begin{equation*}\label{eq:tail_bound}
		\|\bm{x}^*_{\mathcal{S}^c}\|_1 \le \frac{\Gamma^2}{2\lambda\gamma(1-\alpha)\phi_\beta^2(s,s)}.
	\end{equation*}
	Define the extended support $\mathcal{N} := \mathcal{S} \cup \mathcal{S}(\bm{x}^*; s)$. By Lemma 7 in \cite{hu2017group}, one has
	\begin{equation}\label{eq:Nc_bound}
		\|\hat{\bm{x}}_{\mathcal{N}^c}\|_2 \le \frac{1}{\sqrt{s}}\|\bm{x}^*_{\mathcal{S}^c}\|_1 \le \frac{\Gamma^2}{2\sqrt{s}\lambda\gamma(1-\alpha)\phi_\beta^2(s,s)}.
	\end{equation}
	For the component on $\mathcal{N}$, the DCEN-REC$(s, t, \beta)$ condition implies
	\begin{equation}\label{eq:N_bound}
		\|\hat{\bm{x}}_{\mathcal{N}}\|_2 \le \frac{\|\bm{A}\hat{\bm{x}}\|_2}{\phi_\beta(s,s)} \le \frac{2\Gamma}{\phi_\beta^2(s,s)},
	\end{equation}
	where the last inequality utilizes the bound $\|\bm{A}\hat{\bm{x}}\|_2 \le \frac{2\Gamma}{\phi_\beta(s,s)}$ derived from \eqref{eq:proof_mid2}. 
	Summing the squared norms \eqref{eq:Nc_bound} and \eqref{eq:N_bound} due to disjoint support, we have
	\begin{equation*}
		\begin{aligned}
			\|\bm{x}^* - \bar{\bm{x}}\|_2^2 &= \|\hat{\bm{x}}_{\mathcal{N}}\|_2^2 + \|\hat{\bm{x}}_{\mathcal{N}^c}\|_2^2 \\
			&\le \frac{\Gamma^2}{\phi_\beta^4(s,s)} \left( 4 + \frac{\Gamma^2}{4s \lambda^2 \gamma^2 (1-\alpha)^2} \right),
		\end{aligned}
	\end{equation*}
	which establishes \eqref{eq:recovery}. The proof is complete.
\end{proof}

\begin{remark}[]
	Compared to the $\ell_{1-2}$ model ($\alpha=1, \gamma=1$), the proposed Theorem \ref{thm:main_bound} mathematically demonstrates a highly sophisticated trade-off:
	\begin{enumerate}
		\item \textbf{Elimination of Mathematical Singularity:} By restricting $\alpha \in (0,1)$ and adding $\ell_2^2$, DCEN provides a strict coercivity bound. This completely eliminates the fatal mathematical singularity (the $\frac{1}{\sqrt{s}-2}$ blow-up) present in the original $\ell_{1-2}$ theory, guaranteeing a globally continuous and structurally stable recovery bound for \textit{any} arbitrary sparsity level $s \ge 1$.
		\item \textbf{Enhanced Robustness via the Grouping Effect:} The incorporation of the $\ell_2^2$ penalty (controlled by $1-\gamma$) structurally regularizes the Hessian matrix. This effectively expands the restricted cone (allowing a larger $\beta$), granting the model immense robustness against highly coherent design matrices where pure nonconvex methods fail. 
	\end{enumerate}
\end{remark}

\section{Proximal Operator}\label{sec:proxm_operator}

In this section, we derive a closed-form expression for the proximal operator of the DCEN regularization defined as
\begin{equation}\label{eq:proxm_operator}
	\operatorname{prox}_{\lambda r_{\gamma}}(\bm{y})
	=\arg\min_{\bm{x}\in\mathbb{R}^n} f_{\gamma,\alpha}(\bm{x}) =
	r(\bm{x}) +\frac{1}{2\lambda}\|\bm{x}-\bm{y}\|_2^2,
\end{equation}
with $\lambda>0$. To this end, we first show the following existence of solutions of \eqref{eq:proxm_operator}. 

\begin{theorem}[Solution's Existence]
	Let $\lambda>0$, $\gamma\in(0,1)$, $\alpha\in(0,1)$, and $\bm{y}\in\mathbb{R}^n$ be given. Then the solution set of the minimization problem \eqref{eq:proxm_operator} is nonempty, i.e., $\operatorname{prox}_{\lambda r_{\gamma}}(\bm{y}) \neq \emptyset$.
\end{theorem}

\begin{proof}
	Clearly, the objective function in \eqref{eq:proxm_operator} is proper and closed. It follows from the inequality $\|\bm{x}\|_1 \ge \|\bm{x}\|_2$ that $f_{\gamma,\alpha}(\bm{x}) \ge \left(1-\gamma+\frac{1}{2\lambda}\right) \|\bm{x}\|_2^2 + \left[\gamma(1- \alpha) - \frac{1}{\lambda}\|\bm{y}\|_2\right] \|\bm{x}\|_2 + \frac{1}{2\lambda}\|\bm{y}\|_2^2$. The right-hand side is a quadratic function of $\|\bm{x}\|_2$ with a positive leading coefficient $1 - \gamma + \frac{1}{2\lambda} > 0$. Hence, $f_{\gamma,\alpha}(\bm{x}) \to +\infty$ as $\|\bm{x}\|_2 \to \infty$, which implies that $f_{\gamma,\alpha}$ is coercive. Since $f_{\gamma,\alpha}$ is also lower semicontinuous, its sublevel sets are compact and so $\operatorname{prox}_{\lambda r_{\gamma}}(\bm{y}) \neq \emptyset$.
\end{proof}

\begin{lemma}\label{lem:proxm_operator_solu}
	Let $\bm{y}\in\mathbb{R}^n$, $\lambda>0$, $\gamma\in(0,1)$, and $\alpha\ge 0$. 
	Denote by $\bm{x}^*$ an optimal solution to problem~\eqref{eq:unconstprobm}. 
	Then $\bm{x}^*$ admits the following characterization.
	
	\begin{enumerate}
		\item[(a)] If $\|\bm{y}\|_\infty > \gamma\lambda$, then
		$\bm{x}^*=\frac{\|\mathcal{\bm{S}}_{\gamma\lambda}(\bm{y})\|_2 + \alpha\gamma\lambda} {(1+2\lambda(1-\gamma))\,\|\mathcal{\bm{S}}_{\gamma\lambda}(\bm{y})\|_2}\;\mathcal{\bm{S}}_{\gamma\lambda}(\bm{y})$, where $\mathcal{\bm{S}}_{\gamma \lambda}(\bm{y}) = \operatorname{sign}(\bm{y}) \odot \max\big(|\bm{y}| - \gamma \lambda, \, \bm{0}\big)$.
		\item[(b)] If $\|\bm{y}\|_\infty = \gamma \lambda$, then the optimal solution $\bm{x}^*$ satisfies $\bm{x}^* \in \{ \bm{x} \in \mathbb{R}^n : x_i = 0 \text{ for } |y_i| < \gamma \lambda,\ \|\bm{x}\|_2 = \alpha \gamma \lambda,\ x_i y_i \ge 0, \forall i \}$, and $\bm{x}^*$ is non-unique if multiple $|y_i| = \gamma \lambda$.
		\item[(c)] If $(1-\alpha)\gamma \lambda < \|\bm{y}\|_\infty < \gamma \lambda$, then $\bm{x}^*$ is $1$-sparse, supported on an index $i$ with $|y_i| = \|\bm{y}\|_\infty$, and $\bm{x}^* \in \{ \bm{x} \in \mathbb{R}^n\Big| \|\bm{x}\|_2 = \|\bm{y}\|_\infty + (\alpha-1)\gamma \lambda, \ x_i y_i \ge 0 \}.$
		\item[(d)] If $\|\bm{y}\|_\infty \le (1-\alpha)\gamma\lambda$, then $\bm{x}^*=\bm{0}$.
	\end{enumerate}
\end{lemma}
\begin{proof}
	By combining the quadratic terms in the proximal mapping \eqref{eq:proxm_operator}, the resulting operator can be expressed in the same form as the $\ell_1 - \alpha \ell_2$ proximal mapping $\operatorname{prox}_{\lambda r_{\gamma}}(\bm{\tilde{y}}) = \arg\min_{\bm{x}\in \mathbb{R}^n} \, \|\bm{x}\|_1 - \alpha \|\bm{x}\|_2  +\frac{1}{2\tilde{\lambda}} \|\bm{x} - \bm{\tilde{y}}\|_2^2$,	where $\tilde{\lambda} = \frac{\lambda\gamma}{1 + 2 \beta (1 - \gamma)}$ and $\bm{\tilde{y}} =\frac{1}{1 + 2 \beta (1 - \gamma)}\bm{y}$. This reformulation allows one to directly apply the closed-form expressions derived for the $\ell_1 - \alpha \ell_2$ case. The remaining proof follows that of~\cite{lou2018fast} and is therefore omitted.
\end{proof}

\begin{lemma}\label{lem:prox_quad_bound}
	Given $\bm{y} \in \mathbb{R}^n$, $\lambda > 0$, $\gamma \in (0,1)$, and $\alpha \geq 0$. If $\bm{x}^* \in \operatorname{prox}_{\lambda r_{\gamma}}(\bm{y})$ be the minimizer. Then, for any $\bm{x} \in \mathbb{R}^n$, the quadratic upper bound $f_{\gamma,\alpha}(\bm{x}^*) - f_{\gamma,\alpha}(\bm{x})  \leq \min\frac{1}{2}\left( \frac{\gamma \alpha}{\|\bm{x}^*\|_2} - \frac{1}{\lambda} - (1 - \gamma),\; 0 \right) \|\bm{x}^* - \bm{x}\|_2^2$ holds, where $\frac{\gamma \alpha}{0}$ is defined as $0$ if $\alpha = 0$, and $+\infty$ if $\alpha > 0$.
\end{lemma}
\begin{proof}
	We consider two cases based on the support of $\bm{x}^*$.
	
	\textbf{Case 1: $\bm{x}^* \neq \bm{0}$.}  
	Since $\bm{x}^*$ is a minimizer of $f_{\gamma,\alpha}$, according to the first-order optimality condition, we have	 $\bm{0} \in \gamma \partial \|\bm{x}^*\|_1 - \gamma \alpha \frac{\bm{x}^*}{\|\bm{x}^*\|_2} + 2(1 - \gamma)\bm{x}^* + \frac{1}{\lambda}(\bm{x}^* - \bm{y})$. Thus, there exists $\bm{p} \in \partial \|\bm{x}^*\|_1$ such that
	\begin{equation}\label{eq:subgrad_p}
		\bm{p} = \frac{1}{\lambda\gamma} \bm{y} - \left( \frac{1}{\lambda\gamma} + \frac{2(1 - \gamma)}{\gamma} - \frac{\gamma \alpha}{\|\bm{x}^*\|_2} \right) \bm{x}^*.
	\end{equation}
	By the subgradient inequality for $\|\cdot\|_1$, one has $\|\bm{x}\|_1 \geq \|\bm{x}^*\|_1 + \langle \bm{p}, \bm{x} - \bm{x}^* \rangle$, which implies $\|\bm{x}^*\|_1 - \|\bm{x}\|_1 \leq \langle \bm{p}, \bm{x}^* - \bm{x} \rangle$ and so $f_{\gamma,\alpha}(\bm{x}^*) - f_{\gamma,\alpha}(\bm{x})\leq \gamma \langle \bm{p}, \bm{x}^* - \bm{x} \rangle - \gamma\alpha\left(\|\bm{x}^*\|_2-\|\bm{x}\|_2\right) +(1 - \gamma)\left(\|\bm{x}^*\|_2^2-\|\bm{x}\|_2^2\right) + \frac{1}{2\lambda} \left(\|\bm{x}^* - \bm{y}\|_2^2 - \|\bm{x} - \bm{y}\|_2^2\right)$. Substitute the expression for $\bm{p}$ from \eqref{eq:subgrad_p} into the first term, one has
	$
	\gamma \langle \bm{p}, \bm{x}^* - \bm{x} \rangle
	= \gamma \langle \frac{\bm{y} - \bm{x}^*}{\gamma\lambda} -\frac{2(1-\gamma)\bm{x}^*}{\gamma} + \frac{\alpha\bm{x}^*}{\|\bm{x}^*\|_2}, \bm{x}^* - \bm{x}\rangle.
	$
	It follows from $\frac{1}{2\lambda} \left(\|\bm{x}^* - \bm{y}\|_2^2 - \|\bm{x} - \bm{y}\|_2^2\right) = -\frac{1}{2\lambda}\|\bm{x}^*-\bm{x}\|_2^2+\frac{1}{\lambda}\langle \bm{x}^*-\bm{y},\bm{x}^*-\bm{x}\rangle$ that 
	$f_{\gamma,\alpha}(\bm{x}^*) - f_{\gamma,\alpha}(\bm{x})=\gamma\frac{\alpha}{\|\bm{x}^*\|_2}\left(-\langle \bm{x}^*,\bm{x} \rangle + \|\bm{x}^*\|_2\|\bm{x}\|_2 \right) - (1-\gamma)\left(\|\bm{x}^*\|_2^2 +\|\bm{x}\|_2^2\right)-\frac{1}{2\lambda}\|\bm{x}^*-\bm{x}\|_2^2$. From the basic inequalities $\|\bm{x}^*\|_2^2 +\|\bm{x}\|_2^2 \geq \frac{1}{2}\|\bm{x}^*-\bm{x}\|_2^2$ and $\|\bm{x}^*\|_2\|\bm{x}\|_2 \leq \frac{1}{2}\|\bm{x}^*\|_2^2 + \frac{1}{2}\|\bm{x}\|_2^2$, we have 
	\begin{equation*}
		\begin{aligned}
			&\gamma\frac{\alpha}{\|\bm{x}^*\|_2}\left(-\langle \bm{x}^*,\bm{x} \rangle + \|\bm{x}^*\|_2\|\bm{x}\|_2 \right)\\
			 & \quad - (1-\gamma)\left(\|\bm{x}^*\|_2^2 +\|\bm{x}\|_2^2\right)-\frac{1}{2\lambda}\|\bm{x}^*-\bm{x}\|_2^2\\ &\quad \quad =\frac{1}{2}\left(\frac{\gamma\alpha}{\|\bm{x}^*\|_2}-\frac{1}{\lambda}-(1-\gamma)\right).
		\end{aligned}
	\end{equation*}
	
	\textbf{Case 2: $\bm{x}^* = \bm{0}$.}  
	It follows from $f_{\gamma,\alpha}(\bm{0}) \leq f_{\gamma,\alpha}(\bm{x})$ that $f_{\gamma,\alpha}(\bm{x}^*) - f_{\gamma,\alpha}(\bm{x}) \leq 0$. In the special case $\alpha = 0$, expanding $f_{\gamma,\alpha}(\bm{0}) - f_{\gamma,\alpha}(\bm{x})$ yields $f_{\gamma,\alpha}(\bm{0}) -f_{\gamma,\alpha}(\bm{x})= \frac{1}{\lambda} \langle \bm{x}, \bm{y} \rangle - \gamma \|\bm{x}\|_1 - \left(1 - \gamma + \frac{1}{2\lambda}\right) \|\bm{x}\|_2^2$. According to statement (d) in Lemma~\ref{lem:proxm_operator_solu}, we have $\|\bm{y}\|_\infty \leq \gamma \lambda$, which implies that $\langle \bm{x}, \bm{y} \rangle \leq \gamma \lambda \|\bm{x}\|_1$ and  $\frac{1}{\lambda} \langle \bm{x}, \bm{y} \rangle - \gamma \|\bm{x}\|_1 \leq 0$. It follows that $f_{\gamma,\alpha}(\bm{0}) - f_{\gamma,\alpha}(\bm{x}) \leq -\left(1 - \gamma + \frac{1}{2\lambda}\right)\|\bm{x}\|_2^2  \leq 0$.
	Combining both cases, since $f_{\gamma,\alpha}(\bm{x}^*) - f_{\gamma,\alpha}(\bm{x}) \leq 0$ always holds, and in Case 1 it is bounded by a quadratic term, the uniform bound is $f_{\gamma,\alpha}(\bm{x}^*) - f_{\gamma,\alpha}(\bm{x}) \leq \min\left( \frac{1}{2}\left(\frac{\gamma\alpha}{\|\bm{x}^*\|_2}-\frac{1}{\lambda}-(1-\gamma)\right),\; 0 \right) \|\bm{x}^* - \bm{x}\|_2^2$.	When $\gamma = 1$, Lemma~\ref{lem:prox_quad_bound} degenerates into Lemma~2 in~\cite{lou2018fast}.
\end{proof}

\section{DC algorithm}\label{sec:DCAalgorithms}
This section presents DC algorithm (DCA) for DCEN and shows its global convergence and linear convergence rate within the K\L{} framework.
\subsection{DC decomposition and iteration framework}
DCA  can be expressed as the difference of two proper and convex functions, i.e., $f(\bm{x}) = g(\bm{x}) - h(\bm{x})$. At each iteration, DCA linearizes the concave part $-h(\bm{x})$ using a subgradient $\bm{y}^k \in \partial h(\bm{x}^k)$ and solves a convex subproblem $\bm{x}^{k+1} \in \arg\min_{\bm{x}\in \mathbb{R}^n} \left\{ g(\bm{x}) - \langle \bm{y}^k, \bm{x} \rangle \right\}$. The standard iterative framework of DCA is given by
\begin{equation*}
	\bm{x}^{k+1} = \arg\min_{\bm{x} \in \mathbb{R}^n} g(\bm{x}) - \left( h(\bm{x}^k) + \langle \bm{y}^k, \bm{x} - \bm{x}^k \rangle \right),\quad \bm{y}^k \in \partial h(\bm{x}^k).
\end{equation*}
The objective function of the unconstrained optimization problem \eqref{eq:unconstprobm} takes the form
\begin{equation}\label{eq:DCdecom}
	\mathcal{H}(\bm{x}) = g(\bm{x}) -  h(\bm{x}),
\end{equation}
where $g(\bm{x}) = \ell(\bm{x}) + \lambda (\gamma \| \bm{x} \|_1 + \frac{3(1-\gamma)}{2}\|\bm{x}\|_2^2)$ and $h(\bm{x}) = \lambda (\alpha \gamma\| \bm{x} \|_2+\frac{1-\gamma}{2}\|\bm{x}\|_2^2) $. \eqref{eq:DCdecom} is a DC function, amenable to optimization via DCA. Let $\bm{d}_{\ell_2}(\bm{x}_k) = \bigl( \alpha \gamma / \|\bm{x}_k\|_2 \allowbreak + 1 - \gamma \bigr) \bm{x}_k$. The update of DCA for minimizing $\mathcal{H}(\bm{x})$ takes the form
\begin{equation}\label{eq:DCAiterframe}
	\bm{x}_{k+1} =
	\begin{cases}
		\arg\min_{\bm{x} \in \mathbb{R}^n} g(\bm{x}), & \text{if } \bm{x}_k = 0, \\
		\arg\min_{\bm{x} \in \mathbb{R}^n} g(\bm{x}) - \lambda \left\langle \bm{x}, \bm{d}_{\ell_2}(\bm{x}_k) \right\rangle, & \text{otherwise}, 
	\end{cases}
\end{equation}
which follows from linearizing the concave term $-\lambda \left(\alpha \gamma\| \bm{x} \|_2+\frac{1-\gamma}{2}\|\bm{x}\|_2^2\right)$ at $\bm{x}_k$. We adopt the effective  stopping criterion $\|\bm{x}_{k+1} - \bm{x}_k\|_2 \leq \epsilon \bigl(1 + \|\bm{x}_k\|_2\bigr)$, where $\epsilon > 0$ is a prescribed tolerance. The detailed procedure for minimizing DCEN (DCEN-DCA) is presented in Algorithm \ref{alg:DCEN-DCA}.
\begin{algorithm}[!ht]
	\caption{DCEN-DCA}
	\label{alg:DCEN-DCA}
	\begin{algorithmic}[1]
		\REQUIRE Initialize $\bm{A},\bm{b}$, $\bm{x}_0$, $\epsilon > 0$, $\lambda >0$, $\alpha \in(0,1)$, and $\gamma \in(0,1)$.
		\FOR{$k = 1$ to $K$}
		\IF{$\bm{x}_k = \bm{0}$}
		\STATE $\bm{x}_{k+1} \gets \arg\min_{\bm{x} \in \mathbb{R}^n} g(\bm{x})$ 
		\ELSE
		\STATE $\bm{x}_{k+1} \gets \arg\min_{\bm{x} \in \mathbb{R}^n} \left\{ g(\bm{x}) - \lambda \left\langle \bm{x}, d_{\ell_2}(\bm{x}_k) \right\rangle \right\}$ 
		\ENDIF
		\ENDFOR
	\end{algorithmic}
\end{algorithm}

\subsection{ADMM solver for the subproblem in DCEN-DCA}
At each iteration of DCEN-DCA, the concave part of the DC decomposition is linearized at the current point, leading to a convex optimization subproblem 
\begin{equation}\label{eq:dca_subprob}
	\bm{x}_{k+1} \in \arg\min_{\bm{x} \in \mathbb{R}^n} g(\bm{x})- \lambda \left\langle \bm{x},   d_{\ell_2}(\bm{x}_k) \right\rangle.
\end{equation}
In the proposed DCEN-DCA iterator framework, this convex subproblem admits a  solution via ADMM. Specifically, we introduce an auxiliary variable  $\bm{z}$ to decouple the $\ell_1$ norm term. The subproblem \eqref{eq:dca_subprob} is then equivalently reformulated as the following constrained optimization problem
\begin{equation*}
	\begin{aligned}
		&\min_{\bm{x},\bm{z} \in \mathbb{R}^n}\; \ell(\bm{x}) + \frac{3\lambda(1-\gamma)}{2}\|\bm{x}\|_2^2
		- \lambda\langle \bm{x}, \bm{d}_{\ell_2}(\bm{x}_k)\rangle + \lambda\gamma\|\bm{z}\|_1 \\ 
		&\;\;\text{s.t.} \;\bm{x} - \bm{z} = \bm{0}.
	\end{aligned}
\end{equation*}

The augmented Lagrangian function with penalty parameter $\rho>0$ and dual variable $\bm{y}$ is given by
$\mathcal{L}_\rho(\bm{x},\bm{z},\bm{y})= \ell(\bm{x}) + \frac{3\lambda(1-\gamma)}{2}\|\bm{x}\|_2^2- \lambda\langle \bm{x}, \bm{d}_{\ell_2}(\bm{x}_k)\rangle + \lambda\gamma\|\bm{z}\|_1 +\langle \bm{y}, \bm{x}-\bm{z}\rangle + \frac{\rho}{2}\|\bm{x} - \bm{z}\|_2^2$. Let $f(\bm{x})=\ell(\bm{x}) + \frac{3\lambda(1-\gamma)}{2}\|\bm{x}\|_2^2$. Then at the $t$-th iteration of ADMM (with the outer DCEN-DCA index $k$ suppressed for clarity), the primal and dual variables are updated according to
\begin{equation}\label{eq:admm_updatesSimplifying}
	\begin{cases}
		\bm{x}_{t+1} = \displaystyle\arg\min_{\bm{x} \in \mathbb{R}^n} \; f(\bm{x}) - \lambda\bm{x}^\top\bm{d}_{\ell_2}(\bm{x}_k)+\bm{y}_k^\top\bm{x} + \frac{\rho}{2}\|\bm{x} - \bm{z}_k\|_2^2,\\
		\bm{z}_{t+1} = \displaystyle\arg\min_{\bm{z}\in \mathbb{R}^n} \; \frac{\rho}{2}\|\bm{x}_{k+1} - \bm{z}\|_2^2 + \lambda\gamma\|\bm{z}\|_1 +\langle \bm{y}_k, \bm{x}_{k+1}-\bm{z}\rangle,\\
		\bm{y}_{t+1} = \bm{y}_t + \rho \left( \bm{x}_{t+1} - \bm{z}_{t+1} \right).
	\end{cases}
\end{equation}

As for the $\bm{x}$-subproblem of \eqref{eq:admm_updatesSimplifying}, the closed form solution is
$\bm{x}_{t+1}= \big(\bm{A}^\top \bm{A} + (3\lambda(1-\gamma) + \rho) \bm{I} \big)^{-1}\left(\bm{A}^\top \bm{b} + \lambda\, \bm{d}_{\ell_2}(\bm{x}_k) - \bm{y}_{t} + \rho\bm{z}_{t})\right)$. Let $\eta = 3\lambda(1-\gamma) + \rho$ and 
$\bm{R}_t = \bm{A}^\top \bm{b} + \lambda\,\bm{d}_{\ell_2}(\bm{x}_k) - \bm{y}_t + \rho \bm{z}_t$.
By applying the Sherman–Morrison–Woodbury (SMW) identity, we have $\bm{x}_{t+1}= \eta^{-1}\bm{R}_t	- \eta^{-2}\bm{A}^\top	\big(\bm{I} + \eta^{-1}\bm{A}\bm{A}^\top\big)^{-1}\bm{A}\bm{R}_t$. For moderate $n$ precompute a Cholesky factorization of the matrix $\bm{A}^\top \bm{A} + (3\lambda(1-\gamma) + \rho) \bm{I}$ and for large-scale problems use conjugate gradient (CG) with warm start \cite{tao2022minimization}. 

For the $\bm{z}$-update $\bm{z}_{t+1} = \displaystyle\arg\min_{\bm{z}\in \mathbb{R}^n} \; \|\bm{z}\|_1 + \frac{\rho}{2\lambda\gamma}\|\bm{z} - (\bm{x}_{t+1}+1/\rho\bm{y}_{t})\|_2^2$, which has a solution 
$	\bm{z}_{t+1}:= \operatorname{sign}(\bm{x}_{t+1} + \rho^{-1}\bm{y}_t)
\odot \max(\left|\bm{x}_{t+1} + \rho^{-1}\bm{y}_t\right|
- \tfrac{\lambda\gamma}{\rho},\, \bm{0})$.

The inner ADMM iteration is terminated at step $t$ when the primal and dual feasibility, 
together with the relative consistency, are simultaneously satisfied
\begin{equation}\label{eq:admm_stop}
	\begin{aligned}
		\| \bm{r}_t \|_2 \leq \varepsilon_{\mathrm{pri}}^t,\quad\| \bm{s}_t \|_2 \leq \varepsilon_{\mathrm{dual}}^t,\; \text{and} \;\; \mathrm{relerr}_t < \varepsilon_{\mathrm{rel}},
	\end{aligned}
\end{equation}
where $\bm{r}_t = \bm{x}_t - \bm{z}_t$ and $\bm{s}_t = \rho(\bm{z}_t - \bm{z}_{t-1})$. The primal and dual residuals, as well as their tolerances, are given by
\begin{subequations}\label{eq:admm_residuals}
	\begin{align}
		&\varepsilon_{\mathrm{pri}}^t = 
		\sqrt{n}\,\varepsilon_{\mathrm{abs}} 
		+ \varepsilon_{\mathrm{rel}} 
		\max\big(\|\bm{x}_t\|_2, \|\bm{z}_t\|_2\big),\\ &\varepsilon_{\mathrm{dual}}^t = 
		\sqrt{n}\,\varepsilon_{\mathrm{abs}} 
		+ \varepsilon_{\mathrm{rel}}\|\bm{y}_t\|_2, \\
		&
		\mathrm{relerr}_t = 
		\frac{\|\bm{x}_t - \bm{z}_t\|_2}
		{\max\big(\|\bm{x}_t\|_2, \|\bm{z}_t\|_2, \varepsilon\big)},
	\end{align}
\end{subequations}
where $\varepsilon_{\mathrm{abs}}$ and $\varepsilon_{\mathrm{rel}}$ denote the absolute and relative tolerances, respectively, and $\varepsilon$ is the machine precision. The above stopping rule follows the standard criterion proposed in \cite{boyd2011distributed}. ADMM for solving the DCEN-DCA subproblem is summarized in Algorithm \ref{alg:ADMMsub}. Since the subproblem is convex, ADMM converges to its global minimizer under standard assumptions; when embedded in DCEN-DCA (outer nonconvex loop) the whole scheme typically converges to a stationary point of the original nonconvex problem under usual DCEN-DCA assumptions \cite{tao1998dc,le2018dc}.
\begin{algorithm}[t]
	\caption{ADMM for DCEN-DCA subproblem \eqref{eq:dca_subprob}.}
	\label{alg:ADMMsub}
	\begin{algorithmic}[1]
		\REQUIRE $\bm{A},\bm{b}, \bm{d}_{\ell_2}(\bm{x}_k)$; parameters $\varepsilon_{\mathrm{abs}}, \varepsilon_{\mathrm{rel}}, \lambda,\rho>0$, $\gamma\in (0,1)$, and $T$ (maximum inner iteration).
		\ENSURE The solution $\bm{x}_{k+1}$ of the DCA subproblem.
		\STATE Initialize $\bm{x}_0 = \bm{y}_0 = \bm{z}_0$
		\STATE Compute $\eta \gets 3\lambda(1 - \gamma) + \rho$
		\FOR{$t = 0$ to $T - 1$}
		\IF{the stopping condition \eqref{eq:admm_stop} is not met}
		\STATE Compute $\bm{R}_t \gets \bm{A}^\top \bm{b} + \lambda\,\bm{d}_{\ell_2}(\bm{x}_k) - \bm{y}_t + \rho \bm{z}_t$
		\STATE Update $\bm{x}_{t+1} \gets \eta^{-1}\bm{R}_t - \eta^{-2}\bm{A}^\top \big(\bm{I} + \eta^{-1}\bm{A}\bm{A}^\top\big)^{-1} \bm{A}\bm{R}_t$
		\STATE Update $\bm{z}_{t+1} \gets \mathcal{S}_{\lambda\gamma/\rho}\!\left(\bm{x}_{t+1} + \frac{1}{\rho}\bm{y}_{t}\right)$
		\STATE Update $\bm{y}_{t+1} \gets \bm{y}_t + \rho \left( \bm{x}_{t+1} - \bm{z}_{t+1} \right)$
		\STATE Compute $\varepsilon_{\mathrm{pri}}^t$, $\varepsilon_{\mathrm{dual}}^t$, and $\mathrm{relerr}_t$ via \eqref{eq:admm_residuals}
		\ELSE
		\STATE $\bm{x}_{k+1} \gets \bm{x}_{t}$; \textbf{break loop}
		\ENDIF
		\ENDFOR
	\end{algorithmic}
\end{algorithm}

\subsection{Convergence analysis}
The DC decomposition in \eqref{eq:DCdecom} overcomes the lack of strong convexity in the DC decomposition of the original $\ell_1 - \ell_2$ minimization problem \cite{yin2015minimization}. By introducing an additional $\ell_2^2$ regularization term into the objective function, the function $g(\bm{x})$ becomes strongly convex, and the smoothness of $h(\bm{x})$ is also improved through the inclusion of the $\ell_2^2$ regularization term. This modification not only ensures the convergence of DCEN-DCA (satisfying standard theoretical conditions) \cite{tao1998dc,le2018dc} but also significantly enhances the well-conditioning of the subproblems and the overall robustness of the algorithm, while preserving the sparse-inducing property of the $\ell_1 - \alpha \ell_2$ regularization. Moreover,  each subproblem \eqref{eq:dca_subprob} admits a unique global minimizer due to its strong convexity, which further accelerates the convergence and improves the numerical stability of the iterative process. Assuming that each DCEN-DCA iteration \eqref{eq:DCAiterframe} is solved exactly \cite{yin2015minimization}, we establish the algorithm’s convergence as well as its rate of convergence.
\begin{algorithm}[t]
	\caption{ADMM for minimizing DCEN.}
	\label{alg:siopt_admm}
	\begin{algorithmic}[1]
		\REQUIRE $\bm{A}, \bm{b}$; parameters $\varepsilon_{\mathrm{abs}}, \varepsilon_{\mathrm{rel}}, \lambda, \rho > 0$, $\alpha, \gamma \in (0,1)$, and $K$ (maximum iteration).
		\ENSURE The optimal solution $\bm{x}^*$.
		\STATE Initialize $\bm{x}_0 = \bm{z}_0 = \bm{u}_0$
		\STATE Compute $\eta \gets 3\lambda(1 - \gamma) + \rho$
		\FOR{$k = 0$ to $K - 1$}
		\IF{the stopping condition \eqref{eq:admm_stop} is not met}
		\STATE Compute $\bm{\tilde{b}}_k \gets \bm{A}^\top \bm{b} + \rho(\bm{z}_k - \bm{u}_k)$
		\STATE Update $\bm{x}_{k+1} \gets \rho^{-1}\bm{\tilde{b}}_k - \rho^{-2}\bm{A}^\top \big(\bm{I} + \rho^{-1}\bm{A}\bm{A}^\top\big)^{-1} \bm{A}\bm{\tilde{b}}_k$
		\STATE Update $\bm{z}_{k+1} \gets \mathrm{prox}_{\frac{\lambda}{\rho} r_{\gamma}}\!\left(\bm{x}_{k+1} + \bm{u}_k\right)$
		\STATE Update $\bm{u}_{k+1} \gets \bm{u}_k + \bm{x}_{k+1} - \bm{z}_{k+1}$
		\STATE Compute $\varepsilon_{\mathrm{pri}}^k$, $\varepsilon_{\mathrm{dual}}^k$, and $\mathrm{relerr}_k$ via \eqref{eq:admm_residuals}
		\ELSE
		\STATE $\bm{x}^{*} \gets \bm{x}_{k}$ ; \textbf{break loop}
		\ENDIF
		\ENDFOR
	\end{algorithmic}
\end{algorithm}

\begin{proposition}\label{prop:existence}
	For $\bm{A} \in \mathbb{R}^{m \times n}$ and $\bm{b}\in \mathbb{R}^{n}$, if constants $\gamma, \alpha \in (0,1)$ and $\lambda>0$, then  $\mathrm{argmin}_{\bm{x}\in\mathbb{R}^n} \mathcal{H}(\bm{x})$ is nonempty and compact.
\end{proposition}

\begin{proof}
	The function $\mathcal{H}(\bm{x})$ is continuous on $\mathbb{R}^n$ since each component term is continuous. To prove the nonemptiness, it suffices to show that $\mathcal{H}(\bm{x})$ is coercive, i.e., $\mathcal{H}(\bm{x})\to+\infty$ as $\|\bm{x}\|_2\to\infty$. In fact, for any $\bm{x}\in\mathbb{R}^n$, since $\|\bm{x}\|_1\ge\|\bm{x}\|_2$ and $\alpha\in(0,1)$, we have $\|\bm{x}\|_1 - \alpha\|\bm{x}\|_2 \ge (1-\alpha)\|\bm{x}\|_2$. It follows that $\mathcal{H}(\bm{x})\ge \ell(\bm{x}) + \lambda r(\bm{x}) \ge \lambda(1-\gamma)\|\bm{x}\|_2^2$. For $(1-\gamma)>0$, the right-hand side tends to $+\infty$ as $\|\bm{x}\|_2\to\infty$, which establishes coercivity. Because $\mathcal{H}(\bm{x})$ is coercive and continuous on $\mathbb{R}^n$, the sublevel set $\{\bm{x}\in\mathbb{R}^n: \mathcal{H}(\bm{x})\le c\}$ 
	is closed and bounded for any given constant $c$, and so it is compact.  Thus, the solution set $\mathrm{argmin}_{\bm{x}\in\mathbb{R}^n}\,\mathcal{H}(\bm{x})$ is nonempty and compact.
\end{proof}
\begin{remark}
	It is worth noting that the existence result in Proposition~\ref{prop:existence} does not require any specific assumptions on the matrix $\bm{A}$ or the vector $\bm{b}$. 
	Even if $\bm{A}$ is rank-deficient or $\bm{A}=\bm{0}$, the coercivity of $\mathcal{H}(\bm{x})$ is guaranteed by the quadratic term $\lambda(1-\gamma)\|\bm{x}\|_2^2$, 
	which ensures that $\mathcal{H}(\bm{x})\to+\infty$ as $\|\bm{x}\|_2\to\infty$. 
	However, when $\gamma=1$ and $\alpha \ge 1$ \cite{lou2018fast}, the quadratic term vanishes and $\mathcal{H}(\bm{x})$ may lose coercivity, in which case additional assumptions on $\bm{A}$, $\bm{b}$, and constants $\lambda$ and $\alpha$ are required to guarantee the existence of a minimizer. When $\gamma = 1$ and $\alpha = 1$, the matrix $\bm{A}$ must satisfy $\ker(\bm{A}) \cap \bm{\Sigma}_1 = \{\bm{0}\}$, where $\bm{\Sigma}_1 = \{\bm{x} \in \mathbb{R}^n : \|\bm{x}\|_0 \le 1\}$. Otherwise, if there exists a 1-sparse vector $\bm{x} \in \ker(\bm{A})$, it follows that $\|\bm{x}\|_1 - \|\bm{x}\|_2 = 0$. In such a scenario, the $\ell_1 - \ell_2$ model ceases to be coercive. When $\gamma = 1$ and $\alpha > 1$, assuming that for all possible support sets $\Lambda$ with cardinality $|\Lambda| = s$, the submatrix $\bm{A}_\Lambda$ has full column rank, the objective function $\frac{1}{2} \|\bm{Ax} - \bm{b}\|_2^2 + \lambda(\|\bm{x}\|_1 - \alpha \|\bm{x}\|_2)$ is coercive. However, these conditions do not imply global coercivity; rather, coercivity is maintained only along sparse directions, while the functional remains non-coercive over the entire space $\mathbb{R}^n$. Consequently, for the $\|x\|_1 - \alpha \|x\|_2 \ (\alpha \ge 1)$ model, the boundedness of the sequence must be assumed as a prior during algorithmic convergence analysis. This provides an additional motivation for the proposed DCEN model.
\end{remark}

\begin{lemma}\label{lem:nonopt_zero}
	Let $\bm{A} \in \mathbb{R}^{m \times n}$, $\bm{b} \in \mathbb{R}^m$ with $\bm{b} \neq \bm{0}$, $\lambda > 0$, $\gamma \in (0,1)$, and $\alpha \in (0,1)$. If $\bm{b} \notin ker(\bm{A}^{\top})$ and 
	\begin{equation}\label{eq:cond_nonopt}
		\|\bm{A}^\top \bm{b}\|_2 > \lambda \gamma (\sqrt{n} + \alpha),
	\end{equation}
	then \(\bm{x} = \bm{0}\) is not a stationary point of DCEN.
\end{lemma}

\begin{proof}
	Assume, for contradiction, that \(\bm{x} = \bm{0}\) is a stationary point.  
	Since the objective function $\mathcal{H}(\bm{x})$ is proper and lower semicontinuous, a necessary condition for optimality is that $\bm{0} \in \partial \mathcal{H}(\bm{0})$. The subdifferential of $\mathcal{H}(\bm{x})$ at $\bm{0}$ is given by $\partial \mathcal{H}(\bm{0}) = -\bm{A}^\top \bm{b} + \lambda \gamma \left( \partial \|\cdot\|_1(\bm{0}) - \alpha \, \partial \|\cdot\|_2(\bm{0}) \right)$. Recall that $\partial \|\cdot\|_1(\bm{0}) = \{ \bm{u} \in \mathbb{R}^n : \|\bm{u}\|_\infty \leq 1 \}, \;\partial \|\cdot\|_2(\bm{0}) = \{ \bm{v} \in \mathbb{R}^n : \|\bm{v}\|_2 \leq 1 \}$. Thus, $\bm{0} \in \partial \mathcal{H}(\bm{0})$ implies the existence of $\bm{u}$ and $\bm{v}$ satisfying $\|\bm{u}\|_\infty \leq 1$ and $\|\bm{v}\|_2 \leq 1$ such that $\bm{A}^\top \bm{b} = \lambda \gamma (\bm{u} - \alpha \bm{v})$. Taking the Euclidean norm of both sides and applying the triangle inequality yields $\|\bm{A}^\top \bm{b}\|_2 \leq \lambda \gamma \left( \|\bm{u}\|_2 + \alpha \|\bm{v}\|_2 \right)\leq \lambda \gamma \left( \sqrt{n} \|\bm{u}\|_\infty + \alpha \cdot 1 \right)\leq \lambda \gamma (\sqrt{n} + \alpha)$. This contradicts the assumption \eqref{eq:cond_nonopt}. Therefore, $\bm{0}$ cannot be a stationary point.
\end{proof}

Lemma~\ref{lem:nonopt_zero} indicates that, in most practical scenarios, the optimal solution of DCEN is nonzero. Indeed, condition~\eqref{eq:cond_nonopt} is typically satisfied when the regularization parameter $\lambda$ is small, which is the usual case in sparse learning and signal recovery applications. Only when a strong regularization is imposed, i.e., when $\lambda$ takes a sufficiently large value, the zero vector may be a potential minimizer.

\begin{theorem}[Sufficient descent] \label{thm:Sufficient Decrease}
	Let $\{\bm{x}_k\}$ be the sequence generated by the iteration framework \eqref{eq:DCAiterframe} and parameters $\lambda>0$, $\alpha\in(0,1)$, and $\gamma\in(0,1)$. Define $ \mu_g := \lambda_{\min}(\bm{A}^\top\bm{A}) + 3\lambda(1-\gamma) > 0$ and $c := \tfrac{\mu_g}{2}$. Then, the sequence $\{\bm{x}_k\}$ satisfies the sufficient decrease property
	\begin{equation}\label{eq:Sufficient Decrease}
		\mathcal{H}(\bm{x}_k) - \mathcal{H}(\bm{x}_{k+1}) \ge c\,\|\bm{x}_k - \bm{x}_{k+1}\|_2^2.
	\end{equation} 
\end{theorem}

\begin{proof}
	The quadratic function $g_{\mathrm{quad}}(\bm{x})= \tfrac{1}{2}\|\bm{A x}-\bm{b}\|_2^2+ \tfrac{3\lambda(1-\gamma)}{2}\|\bm{x}\|_2^2$ has Hessian matrix $ \nabla^2 g_{\mathrm{quad}}(\bm{x}) = \bm{A}^\top\bm{A} + 3\lambda(1-\gamma)\bm{I}\succ 0$, 
	implying that $g_{\mathrm{quad}}$ is $\mu_g$-strongly convex.
	Since $\lambda\gamma\|\bm{x}\|_1$ is convex, we know that $g$ is  $\mu_g$-strongly convex. Hence, for any $\bm{x}, \bm{y}\in\mathbb{R}^n$ and $\bm{s}\in\partial g(\bm{y})$, one has
	\begin{equation}\label{eq:strongcvx}
		g(\bm{x}) \ge g(\bm{y})
		+ \langle \bm{s},\,\bm{x}-\bm{y}\rangle
		+ \frac{\mu_g}{2}\|\bm{x}-\bm{y}\|_2^2.
	\end{equation}
	For $\bm{y}\neq0$, the subdifferential of $h$ satisfies $\partial h(\bm{y}) = \lambda(\alpha\gamma\,\bm{y}/\|\bm{y}\|_2 + (1-\gamma)\bm{y}) = \lambda\,\bm{d}_{\ell_2}(\bm{y})$. It follows from the convexity of $h$ that 
	\begin{equation}\label{eq:convex_ieee}
		h(\bm{y}) \ge h(\bm{x}) + \langle \bm{s},\,\bm{y}-\bm{x}\rangle,
		\quad \forall \bm{s}\in\partial h(\bm{x}).
	\end{equation}
	
	When $\bm{x}_k\neq0$, the optimality condition of DCEN-DCA subproblem  gives $\bm{0} \in \partial g(\bm{x}_{k+1}) - \lambda\,\bm{d}_{\ell_2}(\bm{x}_k)$, that is, $\bm{s}_g := \lambda\,\bm{d}_{\ell_2}(\bm{x}_k) \in \partial g(\bm{x}_{k+1})$. Applying \eqref{eq:strongcvx} with $\bm{x}=\bm{x}_k$, $\bm{y}=\bm{x}_{k+1}$, and $\bm{s}=\bm{s}_g$ yields
	\begin{equation}\label{eq:gdiff_ieee}
		g(\bm{x}_k)-g(\bm{x}_{k+1})
		\ge \lambda\langle \bm{d}_{\ell_2}(\bm{x}_k),\,\bm{x}_k-\bm{x}_{k+1}\rangle
		+ \frac{\mu_g}{2}\|\bm{x}_k-\bm{x}_{k+1}\|_2^2.
	\end{equation}
	By \eqref{eq:convex_ieee}, taking
	$\bm{x}=\bm{x}_k$, $\bm{y}=\bm{x}_{k+1}$, and
	$\bm{s}=\lambda\bm{d}_{\ell_2}(\bm{x}_k)\in\partial h(\bm{x}_k)$ gives
	\begin{equation}\label{eq:hdiff_ieee}
		h(\bm{x}_k)-h(\bm{x}_{k+1})
		\le \lambda\langle \bm{d}_{\ell_2}(\bm{x}_k),\,\bm{x}_k-\bm{x}_{k+1}\rangle.
	\end{equation}
	Substituting \eqref{eq:gdiff_ieee} and \eqref{eq:hdiff_ieee} into $\mathcal{H}(x_k) - \mathcal{H}(x_{k+1}) = [g(x_k) - g(x_{k+1})] - [h(x_k) - h(x_{k+1})]$, we obtain
	\begin{equation*}
		\begin{aligned}
			&\mathcal{H}(\bm{x}_k)-\mathcal{H}(\bm{x}_{k+1})\\ 
			&\ge ( \lambda \langle d_{\ell_2}(\bm{x}_k), \bm{x}_k - \bm{x}_{k+1} \rangle + \frac{\mu_g}{2} \|\bm{x}_k - \bm{x}_{k+1}\|_2^2 ) \\
			&\quad - \lambda \langle d_{\ell_2}(\bm{x}_k), \bm{x}_k - \bm{x}_{k+1} \rangle = \frac{\mu_g}{2} \|\bm{x}_k - \bm{x}_{k+1}\|_2^2.
		\end{aligned}
	\end{equation*}

	When $\bm{x}_k=\bm{0}$, the subproblem of DCEN-DCA reduces to
	$\bm{x}_{k+1}=\arg\min g(\bm{x})$, implying
	$\bm{0}\in\partial g(\bm{x}_{k+1})$.
	Applying \eqref{eq:strongcvx} with $\bm{x}=\bm{0}$,
	$\bm{y}=\bm{x}_{k+1}$, and $\bm{s}=\bm{0}$ gives
	\begin{equation*}
		\begin{aligned}
			\mathcal{H}(\bm{0}) - \mathcal{H}(\bm{x}_{k+1}) 
			&\ge g(\bm{0}) - g(\bm{x}_{k+1}) \quad (\text{since } h(\bm{x}_{k+1}) \ge 0) \\
			&\overset{\eqref{eq:strongcvx}}{\ge} \frac{\mu_g}{2} \|\bm{x}_{k+1}\|_2^2 = \frac{\mu_g}{2} \|\bm{x}_k - \bm{x}_{k+1}\|_2^2.
		\end{aligned}
	\end{equation*}
	Thus, one has $\mathcal{H}(x_k) - \mathcal{H}(x_{k+1}) \ge \frac{\mu_g}{2} \|x_k - x_{k+1}\|_2^2$.
\end{proof}

\begin{theorem}\label{thm:subsequenceConvergence}
	Let $\{\bm{x}_k\}$ be the sequence generated by the iteration scheme  \eqref{eq:DCAiterframe} with parameters $\lambda>0$, $\alpha\in(0,1)$, and $\gamma\in(0,1)$. Then the following statements hold:
	\begin{enumerate}
		\item[(a)] $\{\bm{x}_k\}$ is bounded;
		\item[(b)] $\|\bm{x}_{k+1} - \bm{x}_k\|_2 \to 0$ as $k \to \infty$;
		\item[(c)] Every limit point $\bm{x}^*$ of $\{\bm{x}_k\}$ satisfies $\bm{0} \in \partial \mathcal{H}(\bm{x}^*)$.
	\end{enumerate}
\end{theorem}

\begin{proof}
	(a) By the coercivity of $\mathcal{H}$ established in Proposition \ref{prop:existence} and $\mathcal{H}(\bm{x}_{k+1}) \leq \mathcal{H}(\bm{x}_k) \leq \mathcal{H}(\bm{x}_0)$ for any $k \geq 0$, it follows that the sequence ${\bm{x}_k}$ is bounded.
	
	(b) By the sufficient decrease property \eqref{eq:Sufficient Decrease} and Lemma \ref{prop:existence}, we have $\mathcal H(\bm{x}_{k+1}) \le \mathcal H(\bm{x}_k)$ for all $k\geq 0$ and  $\{\mathcal H(\bm{x}_k)\}$  admits a finite infimum
	$\underline{\mathcal H}:=\inf_{\bm{x}\in\mathbb R^n}\mathcal H(\bm{x})>-\infty$. It follows that the monotone sequence $\{\mathcal H(\bm{x}_k)\}$ is bounded below and convergent. Denote its limit by $\mathcal H_\infty$. Sum \eqref{eq:Sufficient Decrease} for $k=0,1,\dots,K-1$ to obtain the telescoping inequality $\mathcal H(\bm{x}_0)-\mathcal H(\bm{x}_K) \ge c\sum_{k=0}^{K-1}\|\bm{x}_k-\bm{x}_{k+1}\|_2^2$.
	Since $\mathcal H(\bm{x}_K)\downarrow\mathcal H_\infty$ and $\mathcal H(\bm{x}_0)-\mathcal H_\infty<+\infty$, it follows that
	$\sum_{k=0}^{\infty}\|\bm{x}_k-\bm{x}_{k+1}\|_2^2 \le \frac{1}{c}\big(\mathcal H(\bm{x}_0)-\mathcal H_\infty\big) < +\infty$,
	which implies $\lim_{k\to\infty}\|\bm{x}_k-\bm{x}_{k+1}\|_2^2 = 0$ and so
	$\|\bm{x}_k-\bm{x}_{k+1}\|_2\to 0$.
	
	(c) Let $\bm{x}^*$ be a limit point of the sequence $\{\bm{x}_k\}$ generated by DCEN-DCA. Then there exists a subsequence $\{\bm{x}_{k_j}\}$ such that $\bm{x}_{k_j} \to \bm{x}^*$ as $j \to \infty$. Since $\|\bm{x}_{k+1} - \bm{x}_k\|_2 \to 0$, it follows that $\bm{x}_{k_j + 1} \to \bm{x}^*$. From the optimality condition of subproblem \eqref{eq:dca_subprob}, for each $k$, there exists $\bm{s}_k \in \partial h(\bm{x}_k)$ such that $\bm{s}_k \in \partial g(\bm{x}_{k+1})$. Since $\{\bm{x}_{k_j}\}$ and $\{\bm{x}_{k_j+1}\}$ are bounded and both $g$ and $h$ are proper, lower semicontinuous, and convex, the subdifferentials $\partial h$ and $\partial g$ are locally bounded \cite[Theorem 24.7]{Rockafellar1970}. Hence, the sequence $\{\bm{s}_{k_j}\}$ is bounded and admits a convergent subsequence (still denoted $\{\bm{s}_{k_j}\}$) with limit $\bm{s}^* \in \mathbb{R}^n$. Because $\partial h$ and $\partial g$  have closed graphs \cite[Theorem 24.4]{Rockafellar1970}, the convergences $\bm{x}_{k_j} \to \bm{x}^*$, $\bm{x}_{k_j+1} \to \bm{x}^*$, and $\bm{s}_{k_j} \to \bm{s}^*$, together with the inclusions $\bm{s}_{k_j} \in \partial h(\bm{x}_{k_j})$ and $\bm{s}_{k_j} \in \partial g(\bm{x}_{k_j+1})$, imply that $\bm{s}^* \in \partial h(\bm{x}^*)$ and $\bm{s}^* \in \partial g(\bm{x}^*)$. Thus, $\bm{0} \in \partial g(\bm{x}^*) - \partial h(\bm{x}^*) = \partial \mathcal{H}(\bm{x}^*)$ and so $\bm{x}^*$ is a stationary point of $\mathcal{H}$.
\end{proof}


\subsection{Global convergence and convergence rate}

This subsection presents the global convergence analysis of DCEN-DCA. We first  show that the sequence $\{\bm{x}^k\}$ generated by DCEN-DCA is convergent to a stationary point of $\mathcal{H}$ under mild assumptions. Our analysis follows  the work due to  \cite{boct2022extrapolated,tao2022minimization,wang2019global,zeng2021analysis}.

\begin{proposition}\label{prop:l1l2_constant_value}
	Let $\{\bm{x}_k\}$ be the sequence generated by DCEN-DCA for minimizing $\mathcal{H}(\bm{x})$. Then the following statements hold:
	\begin{enumerate}
		\item[(a)] The limit $\lim_{k \to \infty} \mathcal{H}(\bm{x}_k) = \zeta$ exists;
		\item[(b)] For any $\hat{\bm{x}} \in \mathcal{A}$, $\mathcal{H}(\hat{\bm{x}}) = \zeta$, where $\mathcal{A}$ is the set of accumulation points of the sequence $\{\bm{x}_k\}$.
	\end{enumerate}
\end{proposition}

\begin{proof}
	(a ) By Theorem \ref{thm:Sufficient Decrease} and Proposition \ref{prop:existence}, $\mathcal{H}(\bm{x}_k)$ is non-increasing and $\mathcal{H}$ is level-bounded. It follows that the limit $\lim_{k \to \infty} \mathcal{H}(\bm{x}_k) = \zeta$ exists. 
	
	(b) The update rule of DCEN-DCA  states that at each iteration $k$, there exists a subgradient $\bm{\xi}_k \in \partial h(\bm{x}_k)$ such that $\bm{x}_{k+1} \in \arg\min_{\bm{x} \in \mathbb{R}^n} \left\{ g(\bm{x}) - \langle \bm{\xi}_k, \bm{x} \rangle \right\}$, which implies $\bm{\xi}_k \in \partial g(\bm{x}_{k+1})$. 	Let $\hat{\bm{x}} \in \mathcal{A}$ be an arbitrary accumulation point. Then there is a subsequence $\{\bm{x}_{k_i}\}$ such that $\lim_{i \to \infty} \bm{x}_{k_i} = \hat{\bm{x}}$.  For any $\bm{\hat{z}} \in \mathbb{R}^n$, we have $g(\bm{x}_{k_i+1}) - \langle \bm{\xi}_{k_i}, \bm{x}_{k_i+1} \rangle \leq g(\bm{\hat{z}}) - \langle \bm{\xi}_{k_i}, \bm{\hat{z}} \rangle$. Choosing $\bm{\hat{z}} = \hat{\bm{x}}$ and rearranging yield
	\begin{equation}\label{eq:optim_cond}
		g(\bm{x}_{k_i+1}) \leq g(\hat{\bm{x}}) + \langle \bm{\xi}_{k_i}, \bm{x}_{k_i+1} - \hat{\bm{x}} \rangle.
	\end{equation}
	Since $\{\bm{x}_{k_i}\}$ and $\{\bm{x}_{k_i+1}\}$ are contained in the compact level set $\mathcal{S}_{\mathcal{H}} = \{\bm{x} : \mathcal{H}(\bm{x}) \leq \mathcal{H}(\bm{x}_0)\}$ and $h$ is convex and continuous on $\mathbb{R}^n$,  $\partial h$ is bounded on $\mathcal{S}_{\mathcal{H}}$. Consequently, the sequence of subgradients $\{\bm{\xi}_{k_i}\}$ is bounded, and thus possesses a convergent subsequence, denotes it again by  $\{\bm{\xi}_{k_i}\}$,  such that $\bm{\xi}_{k_i} \to \bm{\xi}^*$ for some $\bm{\xi}^* \in \mathbb{R}^n$.  Using \eqref{eq:optim_cond}, one has $
	\zeta = \lim_{i \to \infty} \mathcal{H}(\bm{x}_{k_i+1})\leq \limsup_{i \to \infty} \left[ g(\hat{\bm{x}}) - h(\bm{x}_{k_i+1}) + \langle \bm{\xi}_{k_i}, \bm{x}_{k_i+1} - \hat{\bm{x}} \rangle \right]$.
	By the continuity of $g$ and $h$, the term $\langle \bm{\xi}_{k_i}, \bm{x}_{k_i+1} - \hat{\bm{x}} \rangle$ converges to $0$. Therefore, the right-hand side equals $g(\hat{\bm{x}}) - h(\hat{\bm{x}}) = \mathcal{H}(\hat{\bm{x}})$, establishing $\zeta \leq \mathcal{H}(\hat{\bm{x}})$. On the other hand, since $\mathcal{H}$ is the difference of a continuous function $g$ and a lower semicontinuous function $h$, it is lower semicontinuous. Thus, $\mathcal{H}(\hat{\bm{x}}) \leq \liminf_{i \to \infty} \mathcal{H}(\bm{x}_{k_i}) = \zeta$ and so $\mathcal{H}(\hat{\bm{x}}) = \zeta$. As $\hat{\bm{x}}$ was chosen arbitrarily from $\mathcal{A}$, the proof is complete.
\end{proof}
\begin{lemma}[Uniformized K\L{} property {\cite{bolte2014proximal}}]\label{lem:Uniformized KL}
	Suppose that $\mathcal{P}$ is a proper closed function and let $\Theta$ be a compact set on which $\mathcal{P}$
	is constant. If $\mathcal{P}$ satisfies the KL property at every point of $\Theta$, then there exist
	constants $\varepsilon,p>0$ and a function $\phi\in\Phi_p$ such that $\phi'(\mathcal{P}(x)-\mathcal{P}(\hat{x}))\,\mathrm{dist}(0,\partial \mathcal{P}(x)) \ge 1$, for all $\hat{x}\in\Theta$ and all $x$ satisfying $\mathrm{dist}(x,\Theta)<\varepsilon$ and
	$\mathcal{P}(\hat{x})<\mathcal{P}(x)<\mathcal{P}(\hat{x})+p$,
	where $\Phi_p$ denotes the class of concave continuous functions
	$\phi:[0,p)\to\mathbb{R}_+$ with $\phi(0)=0$ that are continuously differentiable on $(0,p)$
	and satisfy $\phi'>0$.
\end{lemma}

Let $\mathcal{B}_{\epsilon}(\bm{0}) = \left\{ \bm{x} : \|\bm{x}\|_2 \leq \epsilon \right\} \subset \mathbb{R}^n$ be the closed ball of radius $\epsilon > 0$ centered at the origin, and let $\mathrm{crit}(\mathcal{H})$ denote the set of all critical points of $\mathcal{H}$. We now establish the global convergence of DCEN–DCA.

\begin{theorem}[Global convergence of DCEN-DCA]
	Let the sequence $\{\bm{x}_k\}$ be generated by DCEN-DCA. For parameters $\alpha \in (0,1)$, $\lambda > 0$, and $\gamma \in (0,1)$, if the subgradients of $g$ and $h$ at $\bm{0}$ satisfy
	$
	(-\bm{A}^{\top} \bm{b} + \partial (\lambda \gamma \|\cdot\|_1)(\bm{0})) \;\cap\; \partial (\lambda \alpha \gamma \|\cdot\|_2)(\bm{0})= \emptyset,
	$
	then the following statements hold:
	\begin{enumerate}
		\item[(a)] $\displaystyle\lim_{k\to\infty}{\rm dist}(\bm{0},\partial \mathcal{H}(\bm{x}_k)) = 0$;
		\item[(b)] The sequence $\{\bm{x}_k\}$ converges to a stationary point of $\mathcal{H}$ and $\sum_{k=1}^\infty \|\bm{x}_k - \bm{x}_{k-1}\|_2 < \infty$.
	\end{enumerate}
\end{theorem}

\begin{proof}
	(a) As $\mathcal{H}$ is bounded below, the sequence $\{\bm{x}_k\}$ is bounded and its set of accumulation points $\mathcal{A}$ is nonempty and compact. The assumption that $\partial g(\bm{0}) \cap \partial h(\bm{0}) = \emptyset$ implies that $\bm{0}$ is not a stationary point of $\mathcal{H}$. By Theorem~\ref{thm:subsequenceConvergence}, no accumulation point of $\{\bm{x}_k\}$ can be $\bm{0}$. Consequently, there exists a sufficiently small $\epsilon > 0$ such that $\mathcal{A} \subseteq \mathrm{crit}(\mathcal{H}) \subseteq B_{\epsilon}^c(\bm{0}) := \{\bm{x} \in \mathbb{R}^n : \|\bm{x}\|_2 > \epsilon\}$. Thus, for any $\delta > 0$, there is $K_0 > 0$ such that $\operatorname{dist}(\bm{x}_k$,  $\mathcal{A}) < \delta$ and $\bm{x}_k \in B_{\epsilon}^c(\bm{0})$  for all $k > K_0$, and so $\mathcal{H}$ is Lipschitz continuously differentiable on the open set $B_{\epsilon}^c(\bm{0})$. The optimality condition of \eqref{eq:dca_subprob} yields $\nabla h(\bm{x}_{k-1}) \in \partial g(\bm{x}_k)$. Since $g$ is convex and $h$ is continuously differentiable on $B_{\epsilon}^c(\bm{0})$, one has $\partial \mathcal{H}(\bm{x}_{k}) = \partial g(\bm{x}_{k}) - \nabla h(\bm{x}_{k})$ and so $\nabla h(\bm{x}_{k-1}) - \nabla h(\bm{x}_{k}) \in \partial \mathcal{H}(\bm{x}_{k})$, which implies $\mathrm{dist}\big(\bm{0}, \partial \mathcal{H}(\bm{x}_{k})\big) \le \|\nabla h(\bm{x}_{k-1}) - \nabla h(\bm{x}_{k})\|_2$. Since $\nabla h$ is Lipschitz continuous on $B_{\epsilon}^c(\bm{0})$ with the Lipschitz constant $ L_h = \lambda \left( \frac{\alpha \gamma}{\epsilon} + 1 - \gamma \right)> 0$, we have $\mathrm{dist}\big(\bm{0}, \partial \mathcal{H}(\bm{x}_{k})\big) \le L_h \|\bm{x}_k - \bm{x}_{k-1}\|_2$.
	
	(b) By part (c) of Theorem~\ref{thm:subsequenceConvergence}, it suffices to establish the convergence of the sequence $\{\bm{x}_k\}$. If there exists $k_0 \geq 1$ such that $\mathcal{H}(\bm{x}_{k_0}) = \zeta$, then the sequence $\{\bm{x}_k\}$ converges in finitely many steps. Hence, we assume, without loss of generality, that $\mathcal{H}(\bm{x}_k) > \zeta$ for all $k \geq 1$.
	Since $\mathcal{H}$ is a K\L{} function on the open set $B_{\epsilon}^c(\bm{0})$, Lemma~\ref{lem:Uniformized KL} and Proposition~\ref{prop:l1l2_constant_value} guarantee that there are  $\tilde{\epsilon} > 0$, $\sigma > 0$, and a continuous concave function $\phi \in \Phi_\sigma$ such that the K\L{} inequality $\phi'(\mathcal{H}(\bm{x}) - \zeta) \cdot \operatorname{dist}(\bm{0}, \partial \mathcal{H}(\bm{x})) \geq 1$ holds for all $\bm{x} \in \mathcal{U}$, where $\mathcal{U}:= \{ \bm{x} \in B_{\epsilon}^c(\bm{0}): \operatorname{dist}\big(\bm{x}, \mathcal{A}\big) < \tilde{\epsilon}\}\;\cap \{ \bm{x} \in B_{\epsilon}^c(\bm{0}): \zeta < \mathcal{H}(\bm{x}) < \zeta + \sigma\}$.
	Since $\mathcal{A}$ is the set of accumulation points of $\{\bm{x}_k\}$ and $\{\bm{x}_k\}$ is bounded, there exists $K_1 > 0$ such that $\operatorname{dist}\big(\bm{x}_k, \mathcal{A}\big) < \tilde{\epsilon}$ for all $k > K_1$. Moreover, since $\mathcal{H}(\bm{x}_k) \to \zeta$, there exists $K_2 > 0$ such that $\zeta < \mathcal{H}(\bm{x}_k) < \zeta + \sigma$ for all $k > K_2$. Letting $\widetilde{K} := \max\{K_0 + 1, K_1, K_2\}$, we conclude that $\bm{x}_k \in \mathcal{U}$ for all $k > \widetilde{K}$. Thus, one has $\phi'\big(\mathcal{H}(\bm{x}_k) - \zeta\big) \cdot \operatorname{dist}\big(\bm{0}, \partial \mathcal{H}(\bm{x}_k)\big) \geq 1$ , for any $k > \widetilde{K}$. Using the concavity of $\phi$ and Theorem \ref{thm:Sufficient Decrease}, for $\forall k > \widetilde{K}$, we have
	\begin{equation*}
		\begin{aligned}
			&\frac{\mu_g}{2} \|\bm{x}_k - \bm{x}_{k-1}\|_2^2\\ &\leq \big[ \phi\big(\mathcal{H}(\bm{x}_{k-1}) - \zeta\big) - \phi\big(\mathcal{H}(\bm{x}_k) - \zeta\big) \big] \cdot \operatorname{dist}\big(\bm{0}, \partial \mathcal{H}(\bm{x}_k)\big)  \\
			&\quad\quad\leq \big[ \phi\big(\mathcal{H}(\bm{x}_{k-1}) - \zeta\big) - \phi\big(\mathcal{H}(\bm{x}_k) - \zeta\big) \big] \cdot L_h \|\bm{x}_k - \bm{x}_{k-1}\|_2.
		\end{aligned}
	\end{equation*}
	Rearranging the above inequality yields
	\begin{equation}\label{eq:teleieq}
		\|\bm{x}_k - \bm{x}_{k-1}\|_2 \leq \frac{2L_h}{\mu_g} \big[ \phi\big(\mathcal{H}(\bm{x}_{k-1}) - \zeta\big) - \phi\big(\mathcal{H}(\bm{x}_k) - \zeta\big) \big].
	\end{equation}
	Summing the inequality \eqref{eq:teleieq} from $k = \widetilde{K}$ to $\infty$ and noting that $\phi \geq 0$, we obtain $\sum_{k=\widetilde{K}}^{\infty} \|\bm{x}_k - \bm{x}_{k-1}\|_2 \leq \frac{2L_h}{\mu_g} \phi\big(\mathcal{H}(\bm{x}_{\widetilde{K}}) - \zeta\big) < \infty$ and so  $\sum_{k=0}^{\infty} \|\bm{x}_k - \bm{x}_{k-1}\|_2 < \infty$.
\end{proof}

\begin{theorem}
	Let $\{\bm{x}_k\}$ be the sequence generated by DCEN-DCA. Define $\Lambda^* := \mathrm{supp}(\bm{x}^*)$ and $\Xi := \{ \bm{x} \in \mathbb{R}^n \mid \mathrm{supp}(\bm{x}) = \Lambda^*,\ \mathrm{sign}(\bm{x}) = \mathrm{sign}(\bm{x}^*) \}$. For parameters $\alpha \in (0,1)$, $\lambda > 0$, and $\gamma \in (0,1)$, if $\bm{b} \notin \mathrm{ker}(A^\top)$ and $\|\bm{A}^\top \bm{b}\|_2 > \lambda \gamma (\sqrt{n} + \alpha)$,  then the following statements hold:
	\begin{enumerate}
		\item[(a)] There exists an index $K$ such that $\bm{x}_k \in \Xi$ when $k \ge K$;
		\item[(b)] If the desingularizing function in the K\L{} inequality in Definition~\ref{lem:Uniformized KL} is given by $\phi(\upsilon)=C\,\upsilon^{1-\theta}$ for some $\theta\in[0,1)$ and $C>0$,  then there exist $\tilde{\gamma} > 0$ and $\epsilon>0$ such that $\{\bm{x}_k\}$ converges linearly when $0 < \gamma < \tilde{\gamma}$ and $\lambda\alpha>\epsilon\sigma_{\min} \left( \bm{A}_{\Lambda^*}^\top \bm{A}_{\Lambda^*}\right)$, i.e., there exist $q > 0$ and $\omega \in (0,1)$ such that $\|\bm{x}_k - \bm{x}_*\|_2 \le q \omega^k.$
	\end{enumerate}
\end{theorem} 

\begin{proof}
	(a) Let $\Lambda^* = \operatorname{supp}(\bm{x}^*)$ and $\omega = \min_{i \in \Lambda^*} |x_i^*| > 0$. Suppose, for contradiction, that $\operatorname{supp}(\bm{x}_k)$ is not eventually equal to $\Lambda^*$. Then there exists a subsequence with constant support $\Lambda \neq \Lambda^*$. Since this subsequence also converges to $\bm{x}^*$, we must have $\Lambda^* = \Lambda$, a contradiction. Thus, $\operatorname{supp}(\bm{x}_k) = \Lambda^*$ for all sufficiently large $k$. Since $\bm{x}_k \rightarrow \bm{x}^*$, there exists $K$ such that $\|\bm{x}_k - \bm{x}^*\|_2 < \omega$ whenever $k \geq K$. For any $i \in \Lambda^*$, $|(\bm{x}_k)_i - x_i^*| \leq \|\bm{x}_k - \bm{x}^*\|_2 < \omega \leq |x_i^*|$, which implies $\operatorname{sign}((\bm{x}_k)_i) = \operatorname{sign}(x_i^*)$. It follows that $\bm{x}_k \in \Xi$ for all $k \geq K$.
	
	(b) To demonstrate that the function $\mathcal{H}$ restricted to the set $\Xi$ satisfies the K\L{} property with exponent $\frac{1}{2}$, we introduce an auxiliary function $\widehat{\mathcal{H}}: \mathbb{R}^{|\Lambda^*|} \to \mathbb{R}$, defined as $\widehat{\mathcal{H}}(\bm{u}) = \frac{1}{2} \|\bm{A}_{\Lambda^*} \bm{u} - \bm{b}\|_2^2 + \lambda \left( \gamma \|\bm{u}\|_1 + (1-\gamma) \|\bm{u}\|_2^2 \right) - \lambda \gamma \alpha \|\bm{u}\|_2$. Observe that for any $\bm{x} \in \Xi $, it holds that $ \widehat{\mathcal{H}}(\bm{x}_{\Lambda^*}) = H(\bm{x}) $. Define the set $\Xi_{\Lambda^*} := \{ \bm{u} \in \mathbb{R}^{|\Lambda^*|} \mid \bm{u} = \bm{x}_{\Lambda^*}, \, \bm{x} \in \Xi \}$. Then $\widehat{\mathcal{H}}$ is continuously differentiable on $\Xi_{\Lambda^*}$ and $\nabla \widehat{\mathcal{H}}(\bm{u}) = \bm{A}_{\Lambda^*}^\top (A_{\Lambda^*} \bm{u} - \bm{b}) + 2\lambda (1-\gamma) \bm{u} + \lambda \gamma \operatorname{sign}(\bm{u}) - \lambda \gamma \alpha \frac{\bm{u}}{\|\bm{u}\|_2}$ for all $\bm{u} \in \Xi_{\Lambda^*}$.
	Let $\epsilon < \omega$. Then it is evident that $\bm{x}^* \in B_{\epsilon}^c(\bm{0})$. Next, we verify the inequality $\| \nabla \widehat{\mathcal{H}}(\bm{x}_{\Lambda^*}) - \nabla \widehat{\mathcal{H}}(\bm{x}_{\Lambda^*}^*) \|_2 \geq \tau \| \bm{x}_{\Lambda^*} - \bm{x}_{\Lambda^*}^* \|_2$ for all $\bm{x} \in \Xi \cap B_{\varepsilon}^c(\bm{0})$, where $\tau > 0$. Indeed, let $B_{\epsilon}^c(\bm{0})_{\Lambda^*} := \{ \bm{u} \in \mathbb{R}^{|\Lambda^*|} \mid \bm{u} = \bm{x}_{\Lambda^*}, \, \bm{x} \in B_{\epsilon}^c(\bm{0}) \}$. Then for any $\bm{u}, \bm{v} \in \Xi_{\Lambda^*} \cap B_{\epsilon}^c(\bm{0})_{\Lambda^*}$, we have $\left\| \frac{\bm{u}}{\|\bm{u}\|_2} - \frac{\bm{v}}{\|\bm{v}\|_2} \right\|_2^2  = \frac{1}{\|\bm{u}\|_2 \|\bm{v}\|_2} \left( 2\|\bm{u}\|_2 \|\bm{v}\|_2 - 2 \langle \bm{u}, \bm{v} \rangle \right) \leq \frac{1}{\|\bm{u}\|_2 \|\bm{v}\|_2} \|\bm{u} - \bm{v}\|_2^2$. Set $\hat{L}=\frac{1}{\epsilon}$. Since $\|\bm{u}\|_2 \geq \epsilon$ and $\|\bm{v}\|_2 \geq \epsilon$, it follows that $\left\| \frac{\bm{u}}{\|\bm{u}\|_2} - \frac{\bm{v}}{\|\bm{v}\|_2} \right\|_2 \leq \frac{1}{\epsilon} \|\bm{u} - \bm{v}\|_2=\hat{L}\|\bm{u} - \bm{v}\|_2$. For any $\bm{x} \in \Xi \cap B_{\epsilon}^c(\bm{0})$, we have
		\begin{equation}
		\begin{aligned}
			& \| \nabla \widehat{\mathcal{H}}(\bm{x}_{\Lambda^*}) - \nabla \widehat{\mathcal{H}}(\bm{x}_{\Lambda^*}^*) \|_2 \\
			& = \Big\| ( \bm{A}_{\Lambda^*}^\top \bm{A}_{\Lambda^*} + 2\lambda (1-\gamma) \bm{I} ) \\
			& \quad \times (\bm{x}_{\Lambda^*} - \bm{x}_{\Lambda^*}^*) \\
			& \quad + \lambda \gamma ( \operatorname{sign}(\bm{x}_{\Lambda^*}) - \operatorname{sign}(\bm{x}_{\Lambda^*}^*) ) \\
			& \quad + \lambda \gamma \alpha \Big( \frac{\bm{x}_{\Lambda^*}}{\|\bm{x}_{\Lambda^*}\|_2} - \frac{\bm{x}_{\Lambda^*}^*}{\|\bm{x}_{\Lambda^*}^*\|_2} \Big) \Big\|_2 \\
			& \overset{(a)}{\geq} \big\| ( \bm{A}_{\Lambda^*}^\top \bm{A}_{\Lambda^*} + 2\lambda (1-\gamma) \bm{I} ) (\bm{x}_{\Lambda^*} - \bm{x}_{\Lambda^*}^*) \big\|_2 \\
			& \quad - \big\| \lambda \gamma \alpha \Big( \frac{\bm{x}_{\Lambda^*}}{\|\bm{x}_{\Lambda^*}\|_2} - \frac{\bm{x}_{\Lambda^*}^*}{\|\bm{x}_{\Lambda^*}^*\|_2} \Big) \big\|_2 \\
			& \geq \Big( \sigma_{\min}( \bm{A}_{\Lambda^*}^\top \bm{A}_{\Lambda^*} + 2\lambda (1-\gamma) \bm{I} ) - \frac{\lambda \gamma \alpha}{\varepsilon} \Big) \\
			& \quad \times \| \bm{x}_{\Lambda^*} - \bm{x}_{\Lambda^*}^* \|_2.
		\end{aligned}
	\end{equation}
	Letting $\bar{\gamma} = \frac{\sigma_{\min} \left( \bm{A}_{\Lambda^*}^\top \bm{A}_{\Lambda^*}\right)+2\lambda }{(\hat{L}\alpha+2)\lambda}$, then for $\alpha\in(0,1), \gamma\in(0,1)$, and $\lambda>0$, one has $\bar{\gamma}>0$. Put  $\tau = \sigma_{\min} \left( \bm{A}_{\Lambda^*}^\top \bm{A}_{\Lambda^*}\right) + 2\lambda (1-\gamma) - \frac{\lambda \gamma \alpha}{\varepsilon}$ with $0<\gamma<\bar{\gamma}<1$. Then $\tau>0$ and 
	\begin{equation}\label{eq:lowerbounderH}
		\| \nabla \widehat{\mathcal{H}}(\bm{x}_{\Lambda^*}) - \nabla \widehat{\mathcal{H}}(\bm{x}_{\Lambda^*}^*) \|_2 \geq \tau \| \bm{x}_{\Lambda^*} - \bm{x}_{\Lambda^*}^* \|_2.
	\end{equation}
	Additionally, we show that $\widehat{\mathcal{H}}$ is Lipschitz continuous on $\Xi_{\Lambda^*} \cap B_{\epsilon}^c(\bm{0})_{\Lambda^*}$. In fact, for any $\bm{u}, \bm{v} \in \Xi_{\Lambda^*} \cap B_{\epsilon}^c(\bm{0})_{\Lambda^*}$, we have
	\begin{align*}
		&\| \nabla \widehat{\mathcal{H}}(\bm{u}) - \nabla \widehat{\mathcal{H}}(\bm{v}) \|_2 
		= \left\| \left( \bm{A}_{\Lambda^*}^\top \bm{A}_{\Lambda^*} + 2\lambda (1-\gamma) \bm{I} \right) (\bm{u} - \bm{v}) \right. \\
		&\quad + \lambda \gamma \left( \operatorname{sign}(\bm{u}) - \operatorname{sign}(\bm{v}) \right) \\
		&\quad \left. + \lambda \gamma \alpha \left( \frac{\bm{u}}{\|\bm{u}\|_2} - \frac{\bm{v}}{\|\bm{v}\|_2} \right) \right\|_2 \\
		&\leq \lambda \gamma \alpha \hat{L} \| \bm{u} - \bm{v} \|_2 + \left\| \left( \bm{A}_{\Lambda^*}^\top \bm{A}_{\Lambda^*} + 2\lambda (1-\gamma) \bm{I} \right) (\bm{u} - \bm{v}) \right\|_2 \\
		&\leq L_{\widehat{\mathcal{H}}} \| \bm{u} - \bm{v} \|_2,
	\end{align*}
	where $L_{\widehat{\mathcal{H}}} := \lambda \gamma \alpha \hat{L} + \sigma_{\max} \left( \bm{A}_{\Lambda^*}^\top \bm{A}_{\Lambda^*} \right) + 2\lambda (1-\gamma)$. Thus, for any $\bm{x} \in \Xi \cap B_{\epsilon}^c(\bm{0})$, one has $\| \nabla \widehat{\mathcal{H}}(\bm{x}_{\Lambda^*}) - \nabla \widehat{\mathcal{H}}(\bm{x}_{\Lambda^*}^*) \|_2 \leq L_{\widehat{\mathcal{H}}} \| \bm{x}_{\Lambda^*} - \bm{x}_{\Lambda^*}^* \|_2$.
	From the quadratic upper bound of $\widehat{\mathcal{H}}$, we obtain 
	$\bigl| \mathcal{H}(\bm{x}) - \mathcal{H}(\bm{x}^*) \bigr| \leq  (\bm{x}_{\Lambda^*} - \bm{x}_{\Lambda^*}^*)^\top \nabla \widehat{\mathcal{H}}(\bm{x}_{\Lambda^*}^*) \allowbreak + \frac{L_{\widehat{\mathcal{H}}}}{2} \| \bm{x}_{\Lambda^*} - \bm{x}_{\Lambda^*}^* \|_2^2$.
	Note that $\bm{x}^*$ is a stationary point of $\mathcal{H}$, which implies $\nabla \widehat{\mathcal{H}}(\bm{x}_{\Lambda^*}^*) = 0$. Therefore, for any $\bm{x} \in \Xi \cap B_{\epsilon}^c(\bm{0})$, we have
	$\left| \mathcal{H}(\bm{x}) - \mathcal{H}(\bm{x}^*) \right| \leq \frac{L_{\widehat{\mathcal{H}}}}{2} \| \bm{x}_{\Lambda^*} - \bm{x}_{\Lambda^*}^* \|_2^2$.	From inequality \eqref{eq:lowerbounderH}, we know that for any $\bm{x} \in \Xi \cap B_{\epsilon}^c(\bm{0})$,
	$\left| \mathcal{H}(\bm{x}) - \mathcal{H}(\bm{x}^*) \right| \leq \hbar \, \| \nabla \widehat{\mathcal{H}}(\bm{x}_{\Lambda^*}) \|_2^2,
	$ where $\hbar = L_{\widehat{\mathcal{H}}} / (2\tau^2)$. Moreover, for $\bm{x} \in \mathcal{X} := \{ \bm{x} \in \mathbb{R}^n \mid \bm{x} \in \Xi \; \cap \; \mathcal{U}(\bm{x}^*) \text{ and } \zeta < \mathcal{H}(\bm{x}) < \zeta + \omega \}$, where $\mathcal{U}(\bm{x}^*)$ is a neighborhood of $\bm{x}^*$ such that $\mathcal{U}(\bm{x}^*) \subseteq B_{\epsilon}^c(\bm{0})$, we have
	$\operatorname{dist}^2(0, \partial \mathcal{H}(\bm{x}))= \inf_{\bm{\xi} \in \partial \mathcal{H}(\bm{x})} \| \bm{\xi} \|_2^2 \geq \inf_{\bm{\xi}  \in \partial \mathcal{H}(\bm{x})} \| \bm{\xi}_{\Lambda^*} \|_2^2 = \| \nabla \widehat{\mathcal{H}}(\bm{x}_{\Lambda^*}) \|_2^2 
	\geq | \mathcal{H}(\bm{x}) - \mathcal{H}(\bm{x}^*) |$ and so $\mathcal{H}$ satisfies the KL property with exponent $\frac{1}{2}$ on $\Xi$.
	
	Finally, we prove the linear convergence rate of DCEN-DCA. Letting $S_k = \sum_{i=k}^{\infty} \| \bm{x}_{i+1} - \bm{x}_i \|_2$, then for any $k\geq \widetilde{K}$, we have
	\begin{equation}
		\begin{aligned}
			S_k & \leq 2 \sum_{i=k}^{\infty} \left( \frac{L_h}{\mu_g} \left( \phi(\mathcal{H}(\bm{x}_{i-1}) - \zeta) - \phi(\mathcal{H}(\bm{x}_i) - \zeta) \right) \right)\\ &\leq \frac{2L_h}{\mu_g} \phi(\mathcal{H}(\bm{x}_{k-1}) - \zeta)
		\end{aligned}
	\end{equation}
	and so $S_k \leq \frac{2L_h}{\mu_g} \phi(\mathcal{H}(\bm{x}_{k-1}) - \zeta)$. From the definition of the desingularizing function $\phi$, one has $C (1-\theta) (\mathcal{H}(\bm{x}_{k-1}) - \zeta)^{-\theta} \operatorname{dist}(0, \partial \mathcal{H}(\bm{x}_{k-1})) \geq 1$. Since $\| \bm{x}_k - \bm{x}_{k-1} \|_2 = S_{k-1} - S_k$, we know that  $(\mathcal{H}(\bm{x}_{k-1}) - \zeta)^{\theta} \leq C L_h (1-\theta) (S_{k-1} - S_k)$. Let $\phi(\mathcal{H}(\bm{x}_{k-1}) - \zeta) = C (\mathcal{H}(\bm{x}_{k-1}) - \zeta)^{1-\theta}$. Then $C (\mathcal{H}(\bm{x}_{k-1}) - \zeta)^{1-\theta} \leq C \left( C L_h (1-\theta) (S_{k-1} - S_k) \right)^{\frac{1-\theta}{\theta}}$ 
	and so $S_k \leq \frac{2L_h}{\mu_g} \phi(\mathcal{H}(\bm{x}_{k-1}) - \zeta) \leq \frac{2L_h}{\mu_g} \cdot C \left( C L_h (1-\theta) (S_{k-1} - S_k) \right)^{\frac{1-\theta}{\theta}}$.
	Letting $C_1 = \frac{2L_h}{\mu_g} \cdot C \left( C L_h (1-\theta) \right)^{\frac{1-\theta}{\theta}}$, then one has $S_k \leq C_1 (S_{k-1} - S_k)^{\frac{1-\theta}{\theta}}$. Substituting $\theta = \frac{1}{2}$ yields 
	$C_1 = \frac{L_h^2C^2}{\mu_g}$ and $S_k \leq \frac{C_1^2 L_h^2}{\mu_g} (S_{k-1} - S_k).$
	Rearranging terms gives $S_k \leq \frac{C_1}{1 + C_1} S_{k-1}$.
	Thus, there exists a $K_3 > \widetilde{K}$ such that
	$\| \bm{x}_k - \bm{x}^* \|_2 \leq \sum_{i=k}^{\infty} \| \bm{x}_{i+1} - \bm{x}_i \|_2 = S_k \leq \frac{C_1}{1 + C_1} S_{k-1} \leq \left( \frac{C_1}{1 + C_1} \right)^{k-K_3} S_{K_3}$  for any $k > K_3$, which establishes linear convergence.
\end{proof}

\section{ADMM for Solving DCEN}\label{sec:admm}

In this section, we develop an ADMM for solving DCEN. To this end, we introduce an auxiliary variable $\bm{z}\in\mathbb{R}^n$ and impose $\bm{x}-\bm{z}=\bm{0}$. The equivalent constrained problem reads
\begin{equation}\label{eq:split_problem}
	\min_{\bm{x},\bm{z} \in \mathbb{R}^n}
	\; \ell(\bm{x}) + \lambda r(\bm{z}),
	\;\text{s.t.} \; \bm{x}-\bm{z}=\bm{0}.
\end{equation}
Then the augmented Lagrangian function associated with \eqref{eq:split_problem} is given by
\begin{equation}\label{eq:aug_lagrangian}
	\mathcal{L}_\rho(\bm{x},\bm{z},\bm{y})
	= \ell(\bm{x})
	+ \lambda r_{\gamma}(\bm{z}) + \bm{y}^\top(\bm{x}-\bm{z})
	+ \frac{\rho}{2}\|\bm{x}-\bm{z}\|_2^2,
\end{equation}
where $\bm{y}$ denotes the dual variable and $\rho>0$ is a penalty parameter. By introducing the scaled dual variable $\bm{u}=\bm{y}/\rho$, we obtain the following ADMM iteration scheme  
\begin{subequations}\label{eq:inneradmm_updates}
	\begin{empheq}[left=\empheqlbrace]{align}
		\bm{x}_{k+1} &= \arg\min_{\bm{x}\in \mathbb{R}^n}\;
		\ell(\bm{x}) + \frac{\rho}{2}\|\bm{x}-(\bm{z}_k-\bm{u}_k)\|_2^2, \label{eq:xstep_siopt} \\
		\bm{z}_{k+1} &= \arg\min_{\bm{z}\in \mathbb{R}^n}\;
		\lambda r_{\gamma}(\bm{z})	+ \frac{\rho}{2}\|\bm{z}-p_k\|_2^2, \label{eq:zstep_siopt} \\
		\bm{u}_{k+1} &= \bm{u}_k + \bm{x}_{k+1}-\bm{z}_{k+1}. \label{eq:ustep_siopt}
	\end{empheq}
\end{subequations}
where $p_k=\bm{x}_{k+1}+\bm{u}_k$. The $\bm{x}$-subproblem \eqref{eq:xstep_siopt} is a quadratic program with the closed-form solution $\bm{x}_{k+1} = (\bm{A}^\top\bm{A}+\rho\bm{I})^{-1}(\bm{A}^\top\bm{b} + \rho(\bm{z}_k-\bm{u}_k))$. Using the SMW identity , the $\bm{x}$-update can be rewritten as $\bm{x}_{k+1} = \frac{1}{\rho}\bm{\tilde{b}}_k - \frac{1}{\rho^2}\bm{A}^\top \bm{M}^{-1}(\bm{A}\bm{\tilde{b}}_k)$, where $\bm{\tilde{b}}_k = \bm{A}^\top\bm{b} + \rho(\bm{z}_k - \bm{u}_k)$ and $\bm{M} = \bm{A}\bm{A}^\top + \rho\bm{I} \in \mathbb{R}^{m\times m}$. The $\bm{z}$–subproblem admits a closed-form proximal solution. Specifically, it is updated as $\bm{z}_{k+1} = \mathrm{prox}_{\lambda/\rho r_{\gamma}}\left(\bm{x}_{k+1} + \bm{u}_k\right)$. We summarize the procedure in Algorithm \ref{alg:siopt_admm}. For more discussions on the convergence analysis of the algorithm, interested readers are referred to \cite{tao2022minimization, wang2022minimizing}.

\section{Numerical experiments}\label{sec:NumerExpms}

In this section, we validate the effectiveness of the DCEN method through numerical experiments on sparse signal recovery, MRI image reconstruction, and high-dimensional variable selection. All computational experiments were carried out on a PC equipped with an Intel Core i9\textendash12900H processor (2.50\,GHz), and the simulation environment was implemented on the MATLAB R2025b platform. We compare the DCEN model with the state-of-the-art models and algorithms in sparse recovery, including $\ell_1$-ADMM, $\ell_p (p=1/2)$, IRLSLp$ (p=1/2)$, $\ell_{1}/\ell_{\infty}$-FISTA, $\ell_{1}/\ell_{\infty}$-ADMM, $\ell_1-\alpha\ell_2$-ADMM, $\ell_1-\alpha\ell_2$-IAADMM ($\alpha$ is iteratively adjusted), and $\ell_1-\alpha\ell_2$-DCA, for noiseless and noisy cases in subsections \ref{sec:Noisefree} and \ref{sec:Noise}, respectively. 

\subsection{Parameter settings}\label{sec:Parameterset}

To comprehensively evaluate the performance of the algorithm, we consider two types of sensing matrices $\bm{A} \in \mathbb{R}^{m \times n}$ that are widely used in the compressed sensing literature.
\begin{itemize}
	\item \textbf{Oversampled Discrete Cosine Transform (DCT) Matrix:} Following the construction methods in \cite{rahimi2019scale,tao2022minimization,wang2023variant,yin2015minimization}, the column vectors $\bm{a}_j$ are defined as $\bm{a}_j := 1/\sqrt{m} \cdot \cos \left( 2\pi \bm{w} j/F \right), \; j = 1, \dots, n$, where $\bm{w} \in [0, 1]^m$ is a random vector uniformly distributed within the unit hypercube. The parameter $F > 0$ controls the coherence of the matrix; a larger value of $F$ implies a higher correlation between column vectors, thereby increasing the difficulty of the recovery problem.
	
	\item \textbf{Gaussian Matrix:} This class of matrices is generated by sampling from a multivariate normal distribution $\mathcal{N}(0, \bm{\Sigma})$. By the settings in \cite{yin2015minimization}, the elements of the covariance matrix $\bm{\Sigma}$ are given by $\Sigma_{i,j}=\{(1-r)*\bm{I}(i=j)+r\}_{i,j}$, where $r \in (0, 1)$ is the correlation coefficient. A higher value of $r$ indicates strong correlation among observations, which is typically regarded as a challenging scenario for sparse recovery.
\end{itemize}

We set $\epsilon=10^{-6}$ and the maximum number of iterations $K=50$ for  DCEN-DCA in Algorithm \ref{alg:DCEN-DCA}. For the DCEN-DCA subproblem in Algorithm \ref{alg:ADMMsub} and Algorithm \ref{alg:siopt_admm}, we chose parameters $\varepsilon_{\mathrm{abs}}=10^{-6}$, $\varepsilon_{\mathrm{rel}}=10^{-6}$, $\varepsilon = 10^{-6}$, and the maximum number of iterations $T=5n$. The parameter settings for DCA and ADMM, utilized for solving the $\ell_1-\alpha\ell_2$ model, are maintained consistent with those employed in DCEN. The same goes for $\ell_{1}/\ell_{\infty}$-ADMM. For $\ell_{1}/\ell_{\infty}$\text{-FISTA}, the algorithm is terminated once tolerance $\text{tol}=10^{-6}$ is satisfied or when the number of iterations reaches $5n$. All other settings of the algorithm were set to default ones. To obtain a high-quality initial estimate and avoid local minima, we adopt a warm-start strategy. Specifically, following the standard practice adopted in existing literature, we initialize all algorithms using the solution of $\ell_{1}$-ADMM \cite{rahimi2019scale,tao2022minimization,yin2015minimization}. We assess sparse recovery performance using the success rate, defined as the ratio of successful reconstructions to the total number of trials. A trial is considered successful when the relative error between the reconstruction vector $\bm{x}^*$ and the ground truth signal $\bm{x}^\sharp$ satisfies $\frac{\|\bm{x}^* - \bm{x}^\sharp\|_2}{\|\bm{x}^\sharp\|_2} < 10^{-3}$. 

\subsection{Efficiency Evaluation in Noiseless Scenario}\label{sec:Noisefree}

This subsection is dedicated to comparing the efficiency of the DCEN model and its corresponding solver against established methods in a noiseless setting. The regularization parameters for all compared models are uniformly fixed at $\lambda=10^{-7}$, with the exception of the $\ell_p$ model, which employs an adaptive parameter selection strategy. Apart from the specific stopping criteria mentioned in Section \ref{sec:Parameterset}, all remaining algorithmic parameters adopt their respective default settings.

For the oversampled DCT matrix, the dimensions are set to $300 \times 1000$, with oversampling factors $F = 20$ and $F = 30$, corresponding sparsity levels $s = 10$ and $s = 20$, respectively, and a minimum separation between non-zero entries equal to $F$. For the Gaussian random matrix, the dimensions are set to $64 \times 1024$, with correlation parameters $r = 0.2$ and $r = 0.9$ corresponding to sparsity levels $s = 10$ and $s = 15$, respectively. Each experimental configuration is repeated independently for 100 trials. Figure~\ref{fig:RelativeErrorDistributions} compares the recovery accuracy of DCEN against the $\ell_1-\alpha\ell_2$ model across different sensing matrices and coherence levels, reporting the mean, median, and quantiles of the reconstruction error. The results demonstrate that DCEN consistently achieves superior recovery performance.

To enable a more systematic comparison, we unify the matrix dimensions to $64 \times 1024$ for both the oversampled DCT and Gaussian random matrices. The oversampling factor $F$ is set to 20 and 30. The sparsity level $s$ is varied from 2 to 30 in increments of 2. Each configuration is again repeated 100 times. As shown in Figure~\ref{fig:SuccessRate}, DCEN-DCA achieves higher success rates than the $\ell_1-\alpha\ell_2$ model and other competing algorithms in high coherence scenarios.

\begin{figure*}[t]
	\centering
	\caption{Relative error distributions (log10) of LASSO, $\ell_1-\alpha\ell_2$, and DCEN models under DCT/Gaussian sensing matrices with varying $F$, $r$, and $s$.}
	\label{fig:RelativeErrorDistributions}
	
	\begin{subfigure}[t]{0.35\textwidth}  
		\centering
		\includegraphics[width=\textwidth,trim=0 0 0 0, clip]{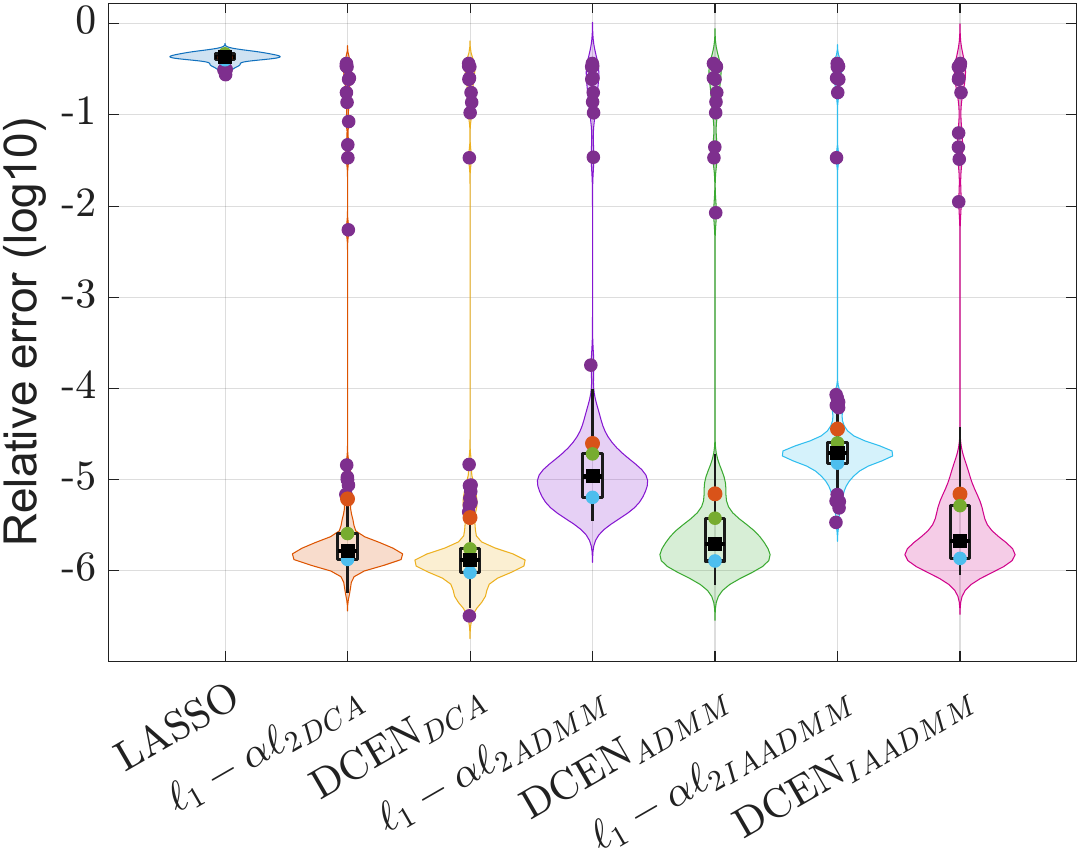}
		\caption{Oversampled DCT: $F=20$ and $s=20$.}
		\label{fig:sub_a}
	\end{subfigure}
	\begin{subfigure}[t]{0.35\textwidth}
		\centering
		\includegraphics[width=\textwidth,trim=0 0 0 0, clip]{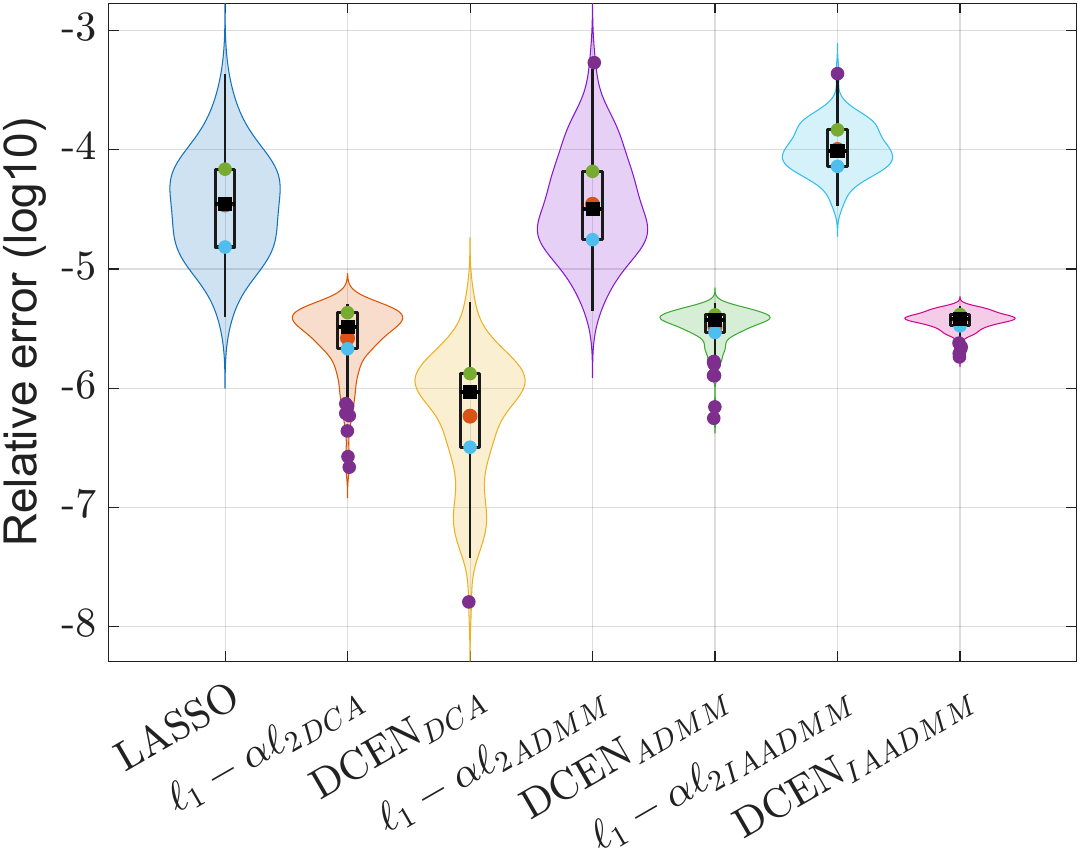}
		\caption{Gaussian: $r=0.2$ and $s=10$.}
		\label{fig:sub_b}
	\end{subfigure}
	
	\vspace{0.8em}  
	
	\begin{subfigure}[t]{0.35\textwidth}
		\centering
		\includegraphics[width=\textwidth]{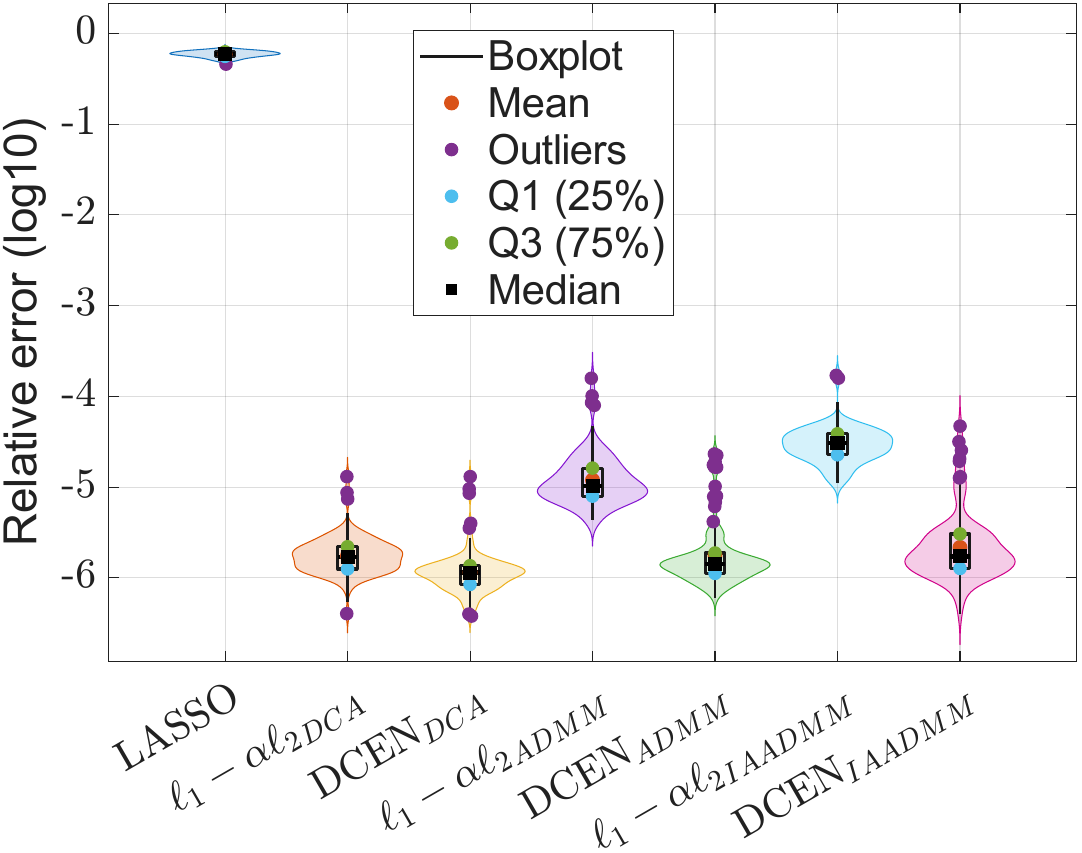}
		\caption{Oversampled DCT: $F=30$ and $s=10$.}
		\label{fig:sub_c}
	\end{subfigure}
	\begin{subfigure}[t]{0.35\textwidth}
		\centering
		\includegraphics[width=\textwidth]{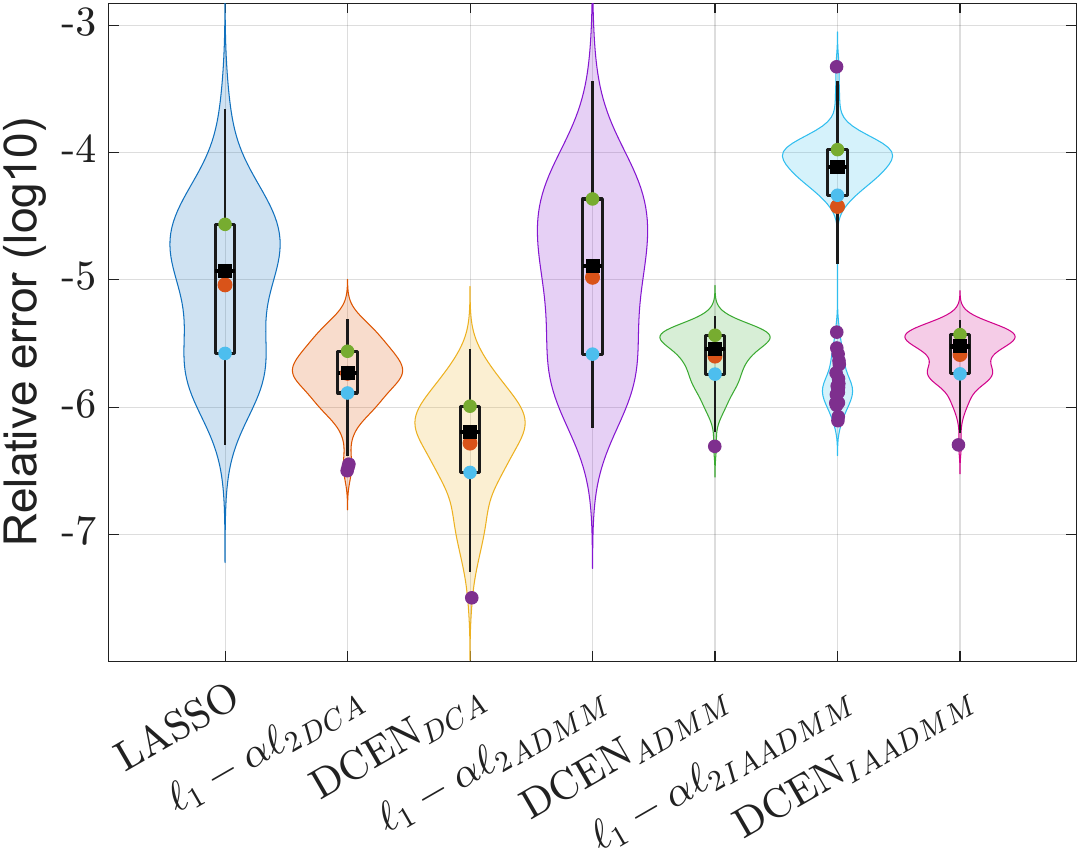}
		\caption{Gaussian: $r=0.9$ and $s=15$.}
		\label{fig:sub_d}
	\end{subfigure}
	
\end{figure*}

%
%
%
%
\begin{figure*}[t]  
	\centering
	\caption{Success rate versus sparsity $s$ for oversampled DCT.}
	\label{fig:SuccessRate}
		\centering
		\includegraphics[width=0.75\textwidth]{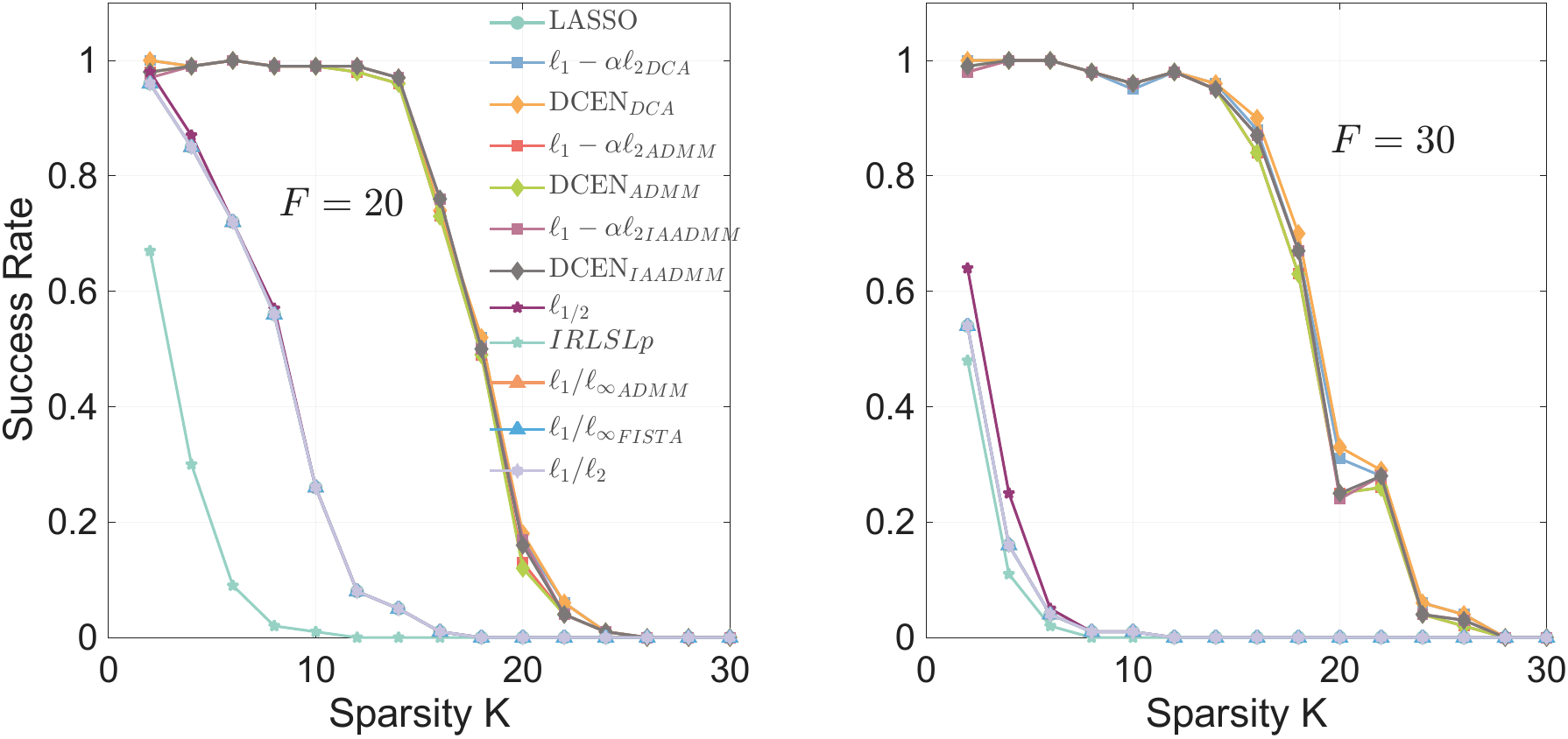}
\end{figure*}

\subsection{Robust recovery of DCEN}\label{sec:Noise}

In this subsection, we evaluate the robust recovery performance of DCEN and its associated solvers—namely, the DCA and ADMM—in noisy scenarios. Specifically, contaminated measurements $\bm{b}$ are generated by adding white Gaussian noise to the clean observation vector $\bm{Ax}^\sharp$, implemented in MATLAB via the \texttt{awgn} function. We consider five signal-to-noise ratio (SNR) levels: 10\,dB, 20\,dB, 30\,dB, 40\,dB, and 50\,dB. For each noise level, we compute the reconstructed solution $\bm{x}^*$ obtained from the $\ell_{1-2}$ and DCEN models and its algorithms, and evaluate the reconstruction SNR as
$10 \log_{10} \frac{\|\bm{x}^\sharp\|_2^2}{\|\bm{x}^* - \bm{x}^\sharp\|_2^2}$.

\begin{table*}
	\centering
	\caption{Reconstruction SNR (dB) for the DCEN and $\ell_1-\alpha\ell_2$ models and their solvers under varying noise levels using an oversampled DCT sensing matrix.}
	\label{tab:SNRmetrics_compare1}
	\resizebox{\textwidth}{!}{%
		\begin{tabular}{cccccccc} 
			\toprule
			SNR(dB) & LASSO    & $\ell_1-\alpha\ell_2$-DCA & DCEN-DCA & $\ell_1-\alpha\ell_2$-ADMM & DCEN-ADMM & $\ell_1-\alpha\ell_2$-IAADMM & DCEN-IAADMM  \\ 
			\hline
			10dB    & 4.6616~  & 6.5908~   & 6.6227~  & 6.5955~    & 6.7078~   & 6.6571~      & \textbf{6.7648}~      \\
			20dB    & 10.9746~ & 19.7750~  & \textbf{20.1652}~ & 19.8222~   & 19.9963~  & 19.9177~     & 20.1323~     \\
			30dB    & 16.2115~ & 28.6825~  & \textbf{28.7654}~ & 28.6514~   & 28.6351~  & 28.6601~     & 28.6668~     \\
			40dB    & 27.1385~ & 38.2164~  & 38.3127~ & 38.2961~   & 38.2931~  & \textbf{38.4896}~     & 38.4493~     \\
			50dB    & 25.4694~ & 37.0743~  & \textbf{37.1253}~ & 37.0910~   & 37.1123~  & 37.0782~     & 37.1143~     \\
			\bottomrule
	\end{tabular}}
\end{table*}

\begin{table*} 
	\centering
	\caption{Reconstruction SNR (dB) for the DCEN and $\ell_1-\alpha\ell_2$ models and their solvers under varying noise levels using a high-coherence Gaussian random sensing matrix ($r=0.9$).}
	\label{tab:SNRmetrics_compare2}
	\resizebox{\textwidth}{!}{%
		\begin{tabular}{cccccccc} 
			\toprule
			SNR(dB) & LASSO   & $\ell_1-\alpha\ell_2$-DCA & DCEN-DCA & $\ell_1-\alpha\ell_2$-ADMM & DCEN-ADMM & $\ell_1-\alpha\ell_2$-IAADMM & DCEN-IAADMM  \\ 
			\hline
			10dB    & 3.9467~ & 5.1931~   & \textbf{5.2826}~  & 5.2179~    & 5.1183~   & 5.2155~      & 5.1155~      \\
			20dB    & 3.8819~ & 5.2219~   & \textbf{5.2695}~  & 5.2362~    & 5.1512~   & 5.2299~      & 5.1456~      \\
			30dB    & 5.6162~ & 8.3980~   & \textbf{8.5392}~  & 8.3845~    & 8.0089~   & 8.3535~      & 7.9788~      \\
			40dB    & 5.7028~ & 8.3714~   & \textbf{8.5469}~  & 8.4014~    & 8.1203~   & 8.3869~      & 8.1189~      \\
			50dB    & 5.6703~ & 8.7965~   & \textbf{8.8369}~  & 8.8048~    & 8.5151~   & 8.7911~      & 8.4710~      \\
			\bottomrule
	\end{tabular}}
\end{table*}

Experiments are conducted on two types of sensing matrices. For the oversampled DCT matrix, we set its dimensions to $100 \times 1000$, with an oversampling factor $F = 15$, sparsity level $s = 10$, and a minimum separation of 15 between non-zero entries. For the Gaussian random matrix, the dimensions are also $100 \times 1000$, with a correlation parameter $r = 0.9$ and sparsity level $s = 30$. Under each experimental configuration, we perform 50 independent trials and report the average reconstruction SNR across all runs. Table~\ref{tab:SNRmetrics_compare1} presents the results for the oversampled DCT matrix. The regularization parameters $\lambda$ and weight parameter $\gamma$ are tuned according to the noise level using a grid search strategy. The results demonstrate that, in most cases, particularly when solved via DCA, DCEN achieves higher reconstruction SNR compared to the $\ell_1-\alpha\ell_2$ model. Table~\ref{tab:SNRmetrics_compare2} reports the results for the Gaussian random matrix. Under high-coherence setting ($r = 0.9$), DCEN-DCA exhibits a clear advantage, consistently outperforming both the $\ell_1-\alpha\ell_2$ model and ADMM-based variants. Notably, DCEN-DCA not only yields higher reconstruction accuracy but also demonstrates superior robustness across varying noise levels. In contrast, the performance of ADMM-type algorithms degrades, likely due to the adverse effects of high matrix coherence on convergence and algorithm parameters tuning. Overall, across all tested scenarios, DCEN-DCA consistently outperforms both ADMM-based solvers and the standard $\ell_1-\alpha\ell_2$ model.

\subsection{MRI Image Reconstruction}

In this subsection, we evaluate the performance of CDCEN  for two-dimensional undersampled magnetic resonance imaging (MRI) reconstruction. Specifically, the DCEN-based reconstruction method is formulated as
\begin{equation}\label{eq:DCENMRI}
	\min_{\bm{u}} \; \gamma \bigl( \|\nabla \bm{u}\|_{1} - \alpha \|\nabla \bm{u}\|_{2} \bigr) + (1 - \gamma) \|\nabla \bm{u}\|_{2}^{2}
	\quad \text{s.t.} \quad \bm{R} \mathcal{F} \bm{u} = \bm{f},
\end{equation}
where $\mathcal{F}$ denotes the two-dimensional discrete Fourier transform, $\bm{R}$ is a radial sampling mask in the frequency domain, and $\bm{f}$ represents the observed $k$-space data. The objective function in \eqref{eq:DCENMRI} can be expressed as the difference of two convex functions; hence, it belongs to the class of DC programs. Consequently, the problem can be efficiently solved by combining DCA with the split Bregman method. At the $k$-th outer DCA iteration, the concave term $-\|\nabla \bm{u}\|_2$ is linearized at the current iterate $\bm{u}_k$, yielding the convex subproblem
\begin{equation}\label{eq:DCENmrisubproblem}
	\begin{aligned}
		\min_{\bm{u}} \; r_1(\bm{u}) + r_2(\bm{u}) - \alpha \gamma \langle \bm{t}^k, \nabla \bm{u} \rangle \; \text{s.t.} \; \bm{R} \mathcal{F} \bm{u} = \bm{f},
	\end{aligned}
\end{equation}
where $r_1(\bm{u}) = \gamma \bigl( |\partial_{\bm{x}} \bm{u}| + |\partial_{\bm{y}} \bm{u}| \bigr)$, $r_2(\bm{u}) = (1 - \gamma) \bigl( |\partial_{\bm{x}} \bm{u}|^2 + |\partial_{\bm{y}} \bm{u}|^2 \bigr)$ and the gradient direction vector $\bm{t}^k = (\bm{t}_{\bm{x}^k}, \bm{t}_{\bm{y}^k})$ is given by
\begin{equation}\label{eq:grad-direc-vec}
	\bm{t}^k = (\partial_{\bm{x}} \bm{u}^k, \partial_{\bm{y}} \bm{u}^k)/\sqrt{|\partial_{\bm{x}} \bm{u}^k|^2 + |\partial_{\bm{y}} \bm{u}^k|^2}.
\end{equation}
To solve subproblem \eqref{eq:DCENmrisubproblem} efficiently, we introduce auxiliary variables $\bm{d}_{\bm{x}} = \partial_{\bm{x}} \bm{u}$, $\bm{d}_{\bm{y}} = \partial_{\bm{y}} \bm{u}$, $\bm{\bar{d}} = (\bm{d}_{\bm{x}}, \bm{d}_{\bm{y}})^\top$ and reformulate the problem \eqref{eq:DCENmrisubproblem} via quadratic penalization as
\begin{equation}\label{eq:DCENmrisubproblem2}
	\begin{aligned}
		&(\bm{u}_{k+1}, \bm{d}_{\bm{x}_{k+1}}, \bm{d}_{\bm{y}_{k+1}}) \\
		&=\arg\min_{\bm{u}, \bm{d_x}, \bm{d_y}} \;  \gamma r_1(\bm{\bar{d}}) + r_2(\bm{\bar{d}}) - \gamma\alpha (\bm{t_x}^k \bm{d_x} + \bm{t_y}^k \bm{d_y}) \\ 
		& + \frac{\mu}{2} \| \bm{R} \mathcal{F} \bm{u} - \bm{f} \|_2^2 + \frac{\beta}{2} \left( \| \bm{d_x} - \partial_{\bm{x}} \bm{u} \|_2^2 + \| \bm{d_y} - \partial_{\bm{y}} \bm{u} \|_2^2 \right).
	\end{aligned}
\end{equation}
By introducing Bregman multipliers $\bm{b} = (\bm{b_x}, \bm{b_y})$ for the constraint $\bm{d} = \nabla \bm{u}$ and a dual variable $\bm{z}$ for the data consistency constraint $\bm{R} \mathcal{F} \bm{u} = \bm{f}$, the split Bregman framework can be applied to solve \eqref{eq:DCENmrisubproblem2} efficiently. The complete algorithmic procedure is summarized in Algorithm~\ref{alg:split_bregman}. Notably, when $\gamma = 1$, the model \eqref{eq:DCENMRI} reduces to the $\ell_1-\alpha\ell_2$ TV-type formulation, and the proposed algorithm simplifies to Algorithm~3 presented in \cite{yin2015minimization}.
\begin{algorithm}[t]
	\caption{Split Bregman method for solving \eqref{eq:DCENMRI}.}
	\label{alg:split_bregman}
	\begin{algorithmic}[1]
		\REQUIRE $\bm{R}, \bm{f}, \mathcal{F}$; parameters $\mu, \beta > 0$; $\gamma, \alpha \in (0,1)$; maximum iterations \texttt{MAXOR} and \texttt{MAXIR}.
		\ENSURE Reconstructed image $\bm{u}^*$.
		\STATE Initialize $\bm{u} \gets \bm{0}$, $\bm{d_x} \gets \bm{0}$, $\bm{d_y} \gets \bm{0}$, $\bm{b_x} \gets \bm{0}$, $\bm{b_y} \gets \bm{0}$, $\bm{z} \gets \bm{f}$
		\FOR{outer = 1 to \texttt{MAXOR}}
		\STATE Update $\bm{t}$ using \eqref{eq:grad-direc-vec}
		\FOR{inner = 1 to \texttt{MAXIR}}
		\STATE $\bm{M}_{R} = \mu \mathcal{F}^\top \bm{R} \bm{z} + \beta \bm{D_x}^\top (\bm{d_x} - \bm{b_x}) + \beta \bm{D_y}^\top (\bm{d_y} - \bm{b_y})$
		\STATE $\bm{u} \gets (\mu \bm{R}^\top \bm{R} - \beta \mathcal{F} \bm{\Delta} \mathcal{F}^\top)^{-1} \bm{M}_{R}$
		\STATE $\bm{d_x} \gets \shrink{ \dfrac{ \gamma \alpha \bm{t_x} + \beta (\bm{D_x u} + \bm{b_x}) }{ \beta + 2(1 - \gamma) } }{ \dfrac{ \gamma }{ \beta + 2(1 - \gamma) } }$
		\STATE $\bm{d_y} \gets \shrink{ \dfrac{ \gamma \alpha \bm{t_y} + \beta (\bm{D_y u} + \bm{b_y}) }{ \beta + 2(1 - \gamma) } }{ \dfrac{ \gamma }{ \beta + 2(1 - \gamma) } }$
		\STATE Update $\bm{b_x} \gets \bm{b_x} + \bm{D_x u} - \bm{d_x}$
		\STATE Update $\bm{b_y} \gets \bm{b_y} + \bm{D_y u} - \bm{d_y}$
		\ENDFOR
		\STATE Update $\bm{z} \gets \bm{z} + \bm{f} - \bm{R} \mathcal{F} \bm{u}$
		\ENDFOR
		\STATE $\bm{u}^* \gets \bm{u}$
	\end{algorithmic}
\end{algorithm}

\begin{figure}[htbp]
	\centering
	\includegraphics[width=0.5\textwidth]{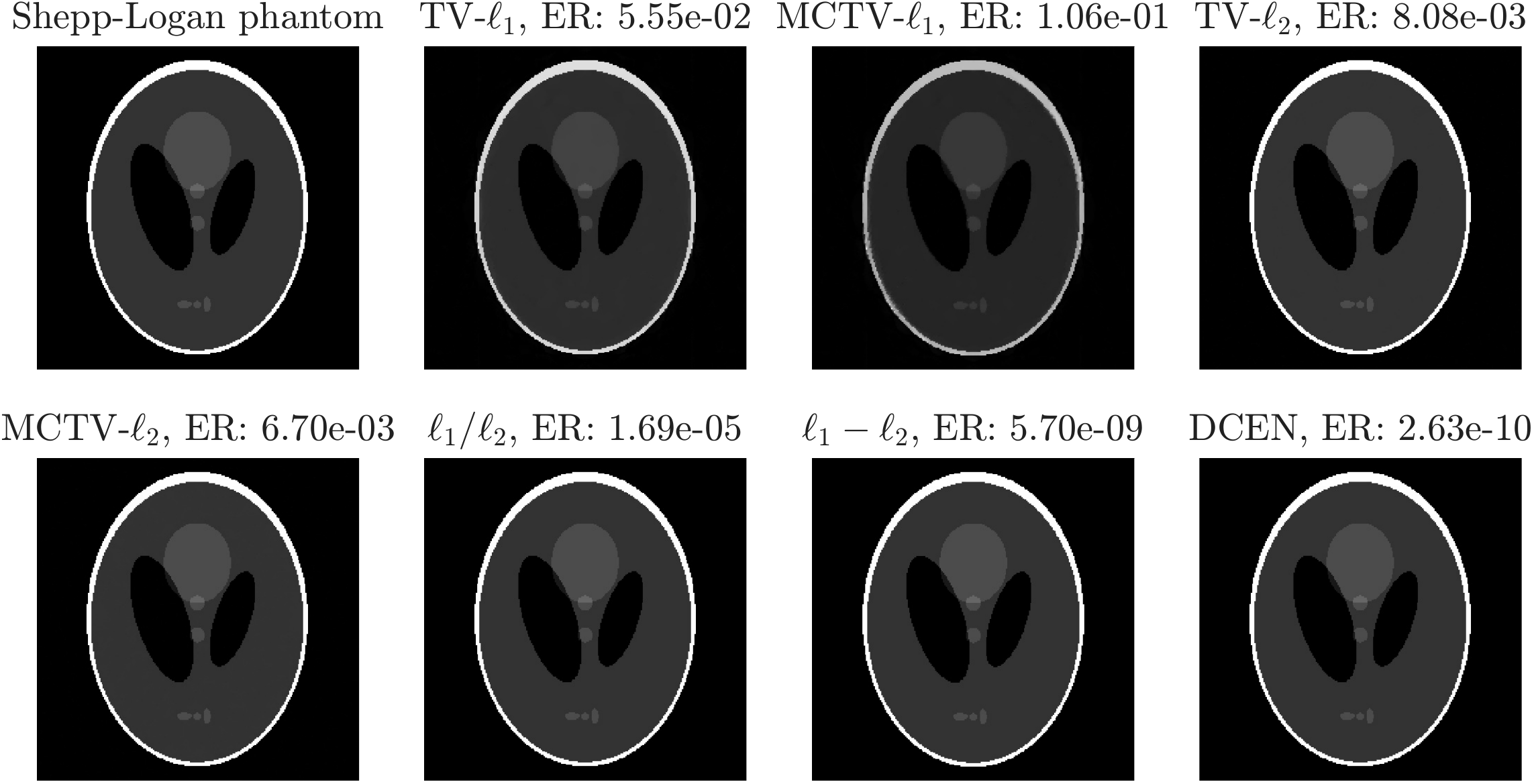}  
	\caption{MRI reconstruction results from 8 radial sampled projections.}
	\label{fig:MRIcon}
\end{figure}
We employ the standard Shepp-Logan phantom as the test image for our experiments. The comparative models include the classical TV-$\ell_1$\cite{candes2006robust}, MCTV-$\ell_1$ \cite{liu2018convex}, TV-$\ell_2$\cite{shen2021magnetic}, MCTV-$\ell_2$\cite{shen2021magnetic}, $\ell_1$-$\ell_2$, and $\ell_1/\ell_2$ models. Figure~\ref{fig:MRIcon} presents the reconstruction accuracy of all models under a radial sampling scheme with only eight projection lines. Experimental results demonstrate that DCEN-TV achieves the highest reconstruction accuracy. Compared to the $\ell_1$-$\ell_2$ model, the relative error of DCEN is reduced by an order of magnitude, indicating a significant improvement in reconstruction quality. Notably, although the $\ell_1/\ell_2$ model is currently regarded as state-of-the-art in MRI reconstruction, it yields substantially higher relative errors than DCEN---highlighting the superior performance of our method in sparse MRI reconstruction.

\subsection{Variable Selection under Highly Correlated Predictors}\label{subsec:correlated_experiment}

To evaluate the feature selection and sparse recovery performance under high correlation, we conduct numerical experiments with $n=20$ and $p=100$. The simulation design incorporates a highly correlated signal block consisting of the first three predictors ($\rho=0.99$), while the remaining 97 variables are i.i.d. standard normal noise. Gaussian noise with $\sigma=1.2$ is added to the response, and all reported metrics are averaged over 100 independent Monte Carlo trials.

The experimental results, summarized in Table~\ref{tab:correlated} and Figure~\ref{fig:SelectVar}, reveal distinct performance gaps among the tested models. Traditional methods, including LASSO and the $\ell_1-\alpha\ell_2$ penalty, exhibit severe degradation in selection probabilities, yielding average true-variable recovery rates of only 29.3\% and 24.3\%, respectively. While the Elastic Net improves the balance across correlated predictors with an average recovery rate of 76.0\%, it does so at the expense of a markedly high false discovery rate, selecting an average of 8.65 noise variables. Furthermore, purely nonconvex penalties such as $\ell_1/\ell_\infty$ and $\ell_1/\ell_2$ demonstrate extreme instability, often concentrating nonzero mass on a single predictor while erroneously discarding the others.

In contrast, the proposed DCEN-ADMM achieves a superior trade-off between sensitivity and specificity. It attains individual true-variable selection probabilities of 85.0\%, 78.0\%, and 86.0\% (averaging 83.0\%), significantly outperforming the Elastic Net. Simultaneously, it maintains a low false discovery rate, with an average of only 4.98 falsely selected variables and a noise selection rate of 5.13\%. Notably, the ADMM-based implementation exhibits greater stability than the DCA-based version in this challenging regime. These results suggest that DCEN-ADMM effectively mitigates the instability inherent in convex models and the local-minima issues of nonconvex approaches, providing robust support recovery for the small-sample, high-dimensional, and strongly correlated data common in bioinformatics and imaging applications.

\begin{table}[tb]
	\centering
	\caption{Variable selection performance under highly correlated predictors. \emph{Avg.\ total}: average number of selected variables; \emph{Avg.\ false}: average number of falsely selected noise variables; \emph{Noise sel.\ rate}: average per-noise-variable selection probability.}
	\footnotesize
	\label{tab:correlated}
	\resizebox{0.5\textwidth}{!}{%
		\begin{tabular}{lcccccc}
			\toprule
			Method & Var1 $\bm{\uparrow}$ & Var2 $\bm{\uparrow}$ & Var3 $\bm{\uparrow}$ & Avg.\ total $\bm{\downarrow}$ & Avg.\ false $\bm{\downarrow}$ & Noise sel.\ rate $\bm{\downarrow}$ \\
			\midrule
			LASSO                  & 44.0\% & 16.0\% & 28.0\% & 6.48  & 5.60 & 5.77\% \\
			Elastic Net            & 78.0\% & 72.0\% & 78.0\% & 10.93 & 8.65 & 8.92\% \\
			$\ell_1-\alpha\ell_2$  & 38.0\% & 11.0\% & 24.0\% & 3.07  & 2.34 & 2.41\% \\
			$\mathrm{DCEN}_{ADMM}$     & 85.0\% & 78.0\% & 86.0\% & 7.47 & 4.98 & 5.13\% \\
			$\mathrm{DCEN}_{DCA}$               & 60.0\% & 22.0\% & 49.0\% & 9.22  & 7.91 & 8.15\% \\
			$\ell_{1/2}$           & 87.0\% &  2.0\% & 19.0\% & 6.49  & 5.41 & 5.58\% \\
			${\ell_1/\ell_\infty}_{FISTA}$ & 100.0\% & 0.0\% & 0.0\% & 8.53 & 7.53 & 7.76\% \\
			$\ell_1/\ell_2$        & 95.0\% & 21.0\% & 72.0\% & 24.06 & 22.18 & 22.87\% \\
			\bottomrule
	\end{tabular}}
\end{table}

\begin{figure}[htbp]
	\centering
	\includegraphics[width=0.5\textwidth]{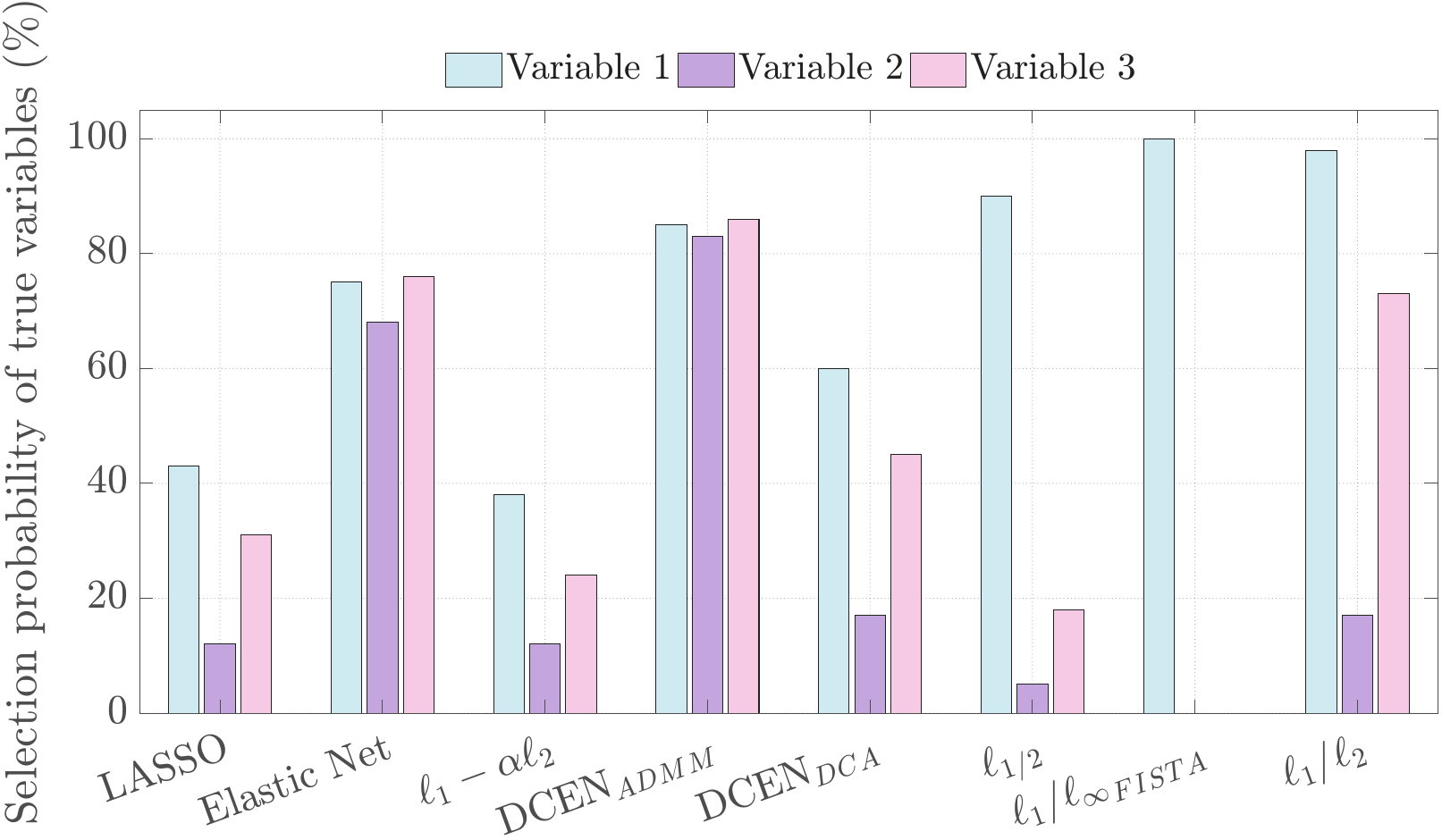}  
	\caption{Variable selection stability under highly correlated predictors ($n=20$, $p=100$; averaged over 100 Monte Carlo trials).}
	\label{fig:SelectVar}
\end{figure}

\section{Conclusion}\label{sec:conclusion}
This work proposes DCEN, a unified Difference-of-Convex modeling framework that integrates the nonconvex $\ell_1 - \alpha\ell_2$ penalty with $\ell_2^2$ regularization. By balancing sparsity promotion with solution stability, DCEN generalizes several representative models—including LASSO, Elastic Net, and the nonconvex $\ell_1-\alpha\ell_2$ model—effectively alleviating $\ell_1$-induced bias while maintaining robustness under high variable correlation.

Theoretically, we establish a comprehensive statistical consistency framework for DCEN. We derive oracle inequalities and provide rigorous recovery bounds that characterize the estimator's behavior in high-dimensional regimes. Furthermore, we establish sufficient conditions for exact and stable recovery under RIP. Algorithmically, we derive closed-form expressions for the proximal operator of the DCEN regularizer, enabling efficient optimization via DCA and ADMM. Within the Kurdyka-Łojasiewicz (K\L) framework, we prove the global convergence of DCA and characterize its linear convergence rate.

Extending the framework to image processing, we develop DCEN-TV using total variation regularization, solved via the split Bregman method. Extensive numerical experiments in high-dimensional variable selection and MRI reconstruction demonstrate that DCEN consistently outperforms state-of-the-art methods in accuracy, robustness, and noise resilience. Future research will explore extensions to low-rank matrix recovery and multitask learning.


 
%

\bibliographystyle{IEEEtran}
\bibliography{refs}

\end{document}